\documentclass[12pt,a4paper,reqno]{amsart}

\usepackage[utf8]{inputenc}
\usepackage{amsthm}
\usepackage{amssymb}
\usepackage[abbrev]{amsrefs}
\usepackage{mathtools}
\usepackage[dvipsnames]{xcolor}
\usepackage{bm}
\usepackage[shortlabels]{enumitem}
\usepackage{hyperref}
\usepackage{bbm}
\usepackage[left=2.5cm,right=2.5cm,top=2cm,bottom=2cm]{geometry}
\newtheorem{thm}{}[section]
\newtheorem{theorem}[thm]{Theorem}

\newtheorem{lemma}[thm]{Lemma}
\newtheorem{proposition}[thm]{Proposition}

\theoremstyle{definition}
\newtheorem{definition}[thm]{Definition}
\theoremstyle{remark}
\newtheorem{remark}[thm]{Remark}

\numberwithin{equation}{section}
\allowdisplaybreaks

\newcommand{\FF}{\ensuremath{\mathbb{F}}}
\newcommand{\A}{\ensuremath{\mathcal{A}}}

\newcommand{\RR}
{\ensuremath{\mathbb{R}}}

\newcommand{\hl}{\ensuremath{\mathbf{h}_l}}
\newcommand{\hr}{\ensuremath{\mathbf{h}_r}}

\newcommand{\NN}{\ensuremath{\mathbb{N}}}
\newcommand{\WW}{\ensuremath{\mathbb{W}}}
\newcommand{\xx}{\ensuremath{\bm{x}}}

\newcommand{\XX}{\ensuremath{\mathbb{X}}}
\newcommand{\XB}{\ensuremath{\mathcal{X}}}

\newcommand{\BB}{\ensuremath{\mathcal{B}}}
\newcommand{\YY}{\ensuremath{\mathbb{Y}}}
\newcommand{\ZZ}{\ensuremath{\mathbb{Z}}}
\newcommand{\Ind}{\ensuremath{\mathbbm{1}}}

\newcommand{\EE}{\ensuremath{\mathbb{E}}}

\newcommand{\Cu}{\ensuremath{\mathcal{Q}}}

\newcommand{\GG}{\ensuremath{\mathcal{G}}}
\newcommand{\CG}{\ensuremath{\mathcal{CG}}}

\def\n{\mathbf{n}}
\DeclareMathOperator{\sgn}{sign}
\DeclareMathOperator{\spn}{span}
\DeclareMathOperator{\supp}{supp}

\newcommand{\floor}[1]{\left\lfloor #1 \right\rfloor}

\AtBeginDocument{\def\MR#1{}}

\author{Miguel Berasategui}

\email{\textcolor{blue}{\href{mailto: mberasategui@dm.uba.ar}{mberasategui@dm.uba.ar}}}
\address{Departamento de Matemática, Universidad de Buenos Aires, (1428) Buenos Aires, Argentina}

\author{Pablo M. Bern\'a}

\email{\textcolor{blue}{\href{mailto:pablo.berna@cunef.edu}{pablo.berna@cunef.edu}}}
\address{Departamento de Matemáticas, CUNEF Universidad, 28040 Madrid, Spain}

\author{Andrea García}

\email{\textcolor{blue}{\href{mailto:andrea.garciapons@usp.ceu.es}{andrea.garciapons@usp.ceu.es}}}
\address{Universidad San Pablo-CEU, CEU Universities, 28003 Madrid, Spain and Departamento de Matemáticas, CUNEF Universidad, 28040 Madrid, Spain}


\title{Approximation spaces, greedy classes and Lorentz spaces}
\thanks{The first author was supported by the Grants ANPCyT PICT 2018-04104 and CONICET PIP 11220200101609CO. The second and third author were supported by the Grant PID2022-142202NB-I00 (Agencia Estatal de Investigación, Spain).}
\begin{document}

\begin{abstract}
The efficiency of greedy algorithms and the democracy functions of bases have been studied in many classical spaces -  such as the $L^p$ spaces or generally the Orlicz spaces $L^\phi (\mathbb R^d)$ -, and have in turn been used to study Jackson and Bernstein inequalities using the classical approximation spaces
and Lorentz spaces.  In this paper, we characterize a generalization of the classical approximation spaces of a broad class of bases - which includes almost greedy bases - in terms of weighted Lorentz spaces. For those bases, we also find necessary and sufficient conditions under which the approximation spaces and greedy and Chebyshev-greedy classes are the same. 
\end{abstract}

\maketitle

\section{Introduction}

Let $\XX$ be a separable and infinite dimensional Banach space with a seminormalized basis $\XB=\left(\xx_n\right)_{n\in \NN}$ and the correspoding dual basis $\XB^*=(\xx_n^*)_{n\in\mathbb N}$. For $0<\alpha, q\le \infty$ the classical non-linear approximation spaces $\A^{\alpha}_q$ are the spaces 
\begin{align*}
\A^{\alpha}_q:=\left\lbrace f\in \XX: \left\Vert f\right\Vert_{\A^{\alpha}_q}<\infty\right\rbrace, 
\end{align*}
where 
\begin{align*}
\left\Vert f\right\Vert_{\A^{\alpha}_q}:=&
\begin{cases}
\left\Vert f\right\Vert_{\XX}+\left(\sum_{n\in \NN}\left(n^{\alpha}\sigma_n\left(f\right)\right)^{q}\frac{1}{n}\right)^{\frac{1}{q}} & \text{ if }0<q<\infty;\\
\left\Vert f\right\Vert_{\XX}+\sup_{n\in \NN}n^{\alpha}\sigma_n\left(f\right) & \text{ if } q=\infty, 
\end{cases}
\end{align*}
and $\sigma_n\left(f\right)$ is \textit{best $n$-term approximation error} of $f$ based on $\XB$ given by 
\begin{align*}
\sigma_n\left(f\right):=&\inf_{g\in\Sigma_n(\XB)}\left\Vert f-g\right\Vert_{\XX},
\end{align*}
where 
$$\Sigma_n(\XB) := \left\lbrace g\in\mathbb X : g=\sum_{n\in A} a_n \xx_n,\; \vert A\vert\leq n\right\rbrace.$$
It is known that the spaces $\mathcal A_q^\alpha$ are quasi-Banach spaces (\cite{P1981}). In approximation theory, it is of interest to find characterizations of these spaces in terms of properties that may be easier to study, as well as to find embeddings from and into Lorentz sequence spaces (see for example \cite{GH2004}, \cite{GHN2012}, \cite{GN2001}, \cite{KP2006}); applications in Harmonic Analysis can be found for example in \cite{GHM}. In this context, in \cite{GN2001} the authors defined the greedy approximation classes $\GG^{\alpha}_q$, which are a modification of the classsical approximation spaces using the Thresholding Greedy Algorigthm, introduced by S. V. Konyagin and V. N. Temlyakov in \cite{KT1999}: given a basis $\XB$ of $\XX$ and $f\in \XX$, a \textit{greedy set of $f$ of cardinality $m$} is a set $A\subset \NN$ with $\left\vert A\right\vert=n$ such that 
\begin{align*}
&\left\vert \xx_n^*\left(f\right)\right\vert\ge  \left\vert \xx_j^*\left(f\right)\right\vert&&\forall n\in A,\;\forall j\not\in A. 
\end{align*}
For $0<\alpha<\infty$ and $0<q\le \infty$, the greedy approximation classes are defined in the same way as $\A^{\alpha}_q$ but with $\gamma_m\left(f\right)$ replaced with 
\begin{align*}
\gamma_m\left(f\right):=&\sup_{A\in GS\left(f,m\right)} \left\Vert f-P_A\left(f\right)\right\Vert_{\XX}, 
\end{align*}
where $GS\left(f,m\right)$ is the set of greedy sets of $f$ of cardinality $m$, and is the proyection of $f$ on the space $\langle \xx_n: n\in A\rangle$, that is $$P_A\left(f\right):=\sum_{n\in A}\xx_n^*\left(f\right)\xx_n.$$
With these definitions, in \cite{GN2001}, we found the following classes:
$$\GG_q^\alpha=\lbrace f\in\mathbb X : \Vert f\Vert_{\GG_q^\alpha}<\infty\rbrace,$$
where $\Vert f\Vert_{\GG_q^\alpha}$ is exactly the same definition than in $\mathcal A_q^\alpha$ replacing $\sigma_m(f)$ by $\gamma_m(f)$.

Unlike the spaces $\A^{\alpha}_q$, the classes $\GG^{\alpha}_q$ need not be linear spaces \cite{GHN2012}, but for a wide class of bases, the two coincide. It is immediate from definitions that $\GG^{\alpha}_q\subset  \A^{\alpha}_q$ and that $\left\Vert \cdot\right\Vert_{\A^{\alpha}_q}\le \left\Vert \cdot\right\Vert_{\GG^{\alpha}_q}$. In the opposite direction, Gribonval and Nielsen pointed out that we have $\GG^{\alpha}_q=  A^{\alpha}_q$ and $\left\Vert \cdot\right\Vert_{\GG^{\alpha}_q}\lesssim \left\Vert \cdot\right\Vert_{\A^{\alpha}_q}$ when the basis is greedy (see \cite{KT1999}), that is when there is $C>0$ such that 
\begin{align}
&\left\Vert f-P_A\left(f\right)\right\Vert_{\XX}\le C\sigma_n\left(f\right)&&\forall f\in \XX, n\in \NN, A\in GS\left(f,n\right). \label{greedy1}
\end{align}
Using the identification of greedy bases from \cite{KT1999} in terms of unconditionality and democracy, where the latter means that there is $C>0$ such that

\begin{align*}
&\left\Vert \sum_{n\in A}\xx_n\right\Vert\le C \left\Vert \sum_{j\in B}\xx_j\right\Vert &&\forall A, B\subset \NN: \left\vert A\right\vert\le \left\vert B \right\vert<\infty, 
\end{align*}
Gribonval and Nielsen proved \cite[Remark 6.2]{GN2001} that the equality  $\GG^{\alpha}_q=A^{\alpha}_q$, even with equivalent quasi-norms - which we will note  $\GG^{\alpha}_q\approx A^{\alpha}_q$ -, does not entail that the basis is greedy, since it holds for some conditional \emph{almost greedy} bases, that is bases for which  \eqref{greedy1} holds if we replace $\sigma_n\left(f\right)$ with the best $n$-term approximation error via proyections, that is 
\begin{align*}
\widetilde{\sigma}_n\left(f\right):=&\inf_{\substack{A\subset \NN\\\left\vert A\right\vert\le n}}\left\Vert f-P_A\left(f\right)\right\Vert_{\XX}. 
\end{align*}
The study of the bases for which $\GG^{\alpha}_q=A^{\alpha}_q$ continued in \cite{GHN2012}, \cite{W2014} and \cite{BCH2023}. In \cite{GHN2012}, the authors showed that under fairly general conditions, for an unconditional basis that fails to be democratic $\GG^{\alpha}_q\approx A^{\alpha}_q$ does not hold, and conjectured that, for unconditional bases, $\GG^{\alpha}_q\approx A^{\alpha}_q$ if and only if the basis is democratic. This conjecture was proved in \cite[Theorem 3.1]{W2014}. Later,  in  \cite[Theorem 2.14]{BCH2023}, this result was extended to spaces where $k^{\alpha}$ is replaced by a wider class of weights. \\
As pointed out in \cite{W2014}, even for unconditional bases the study of the conditions under which $\GG^{\alpha}_q=A^{\alpha}_q$ has so far been undertaken only under the assumption of equivalent quasi-norms. Our first goal is to remove that limitation and characterize the unconditional bases for which $\GG^{\alpha}_q=A^{\alpha}_q$. Moreover, in line with the study of the greedy classes in  \cite{GN2001}, we will relax the unconditionality condition. \\
The study of the condition $\GG^{\alpha}_q=A^{\alpha}_q$ is of interest not only for its own sake, but also as a tool for studying the classical approximation spaces. In this direction, in \cite{GN2001} the authors showed that for almost greedy bases, the space $\GG^{\alpha}_q$ is a quasi-Banach space isomorphic to a weighted Lorentz sequence space \cite[Theorem 5.1]{GN2001}, and used that result to similarly characterize $\A^{\alpha}_q$ for greedy bases. Our second goal is to extend these characterizations beyond greedy bases, to cover almost greedy bases and even further. \\
Our third goal is to replace the hypothesis $\GG^{\alpha}_q=A^{\alpha}_q$ with the weaker hypothesis $\CG^{\alpha}_q=A^{\alpha}_q$, where $\CG^{\alpha}_q$ is the \emph{Chebyshev-greedy approximation class}, introduced in \cite[Theorem 2.14]{BCH2023}, and which is defined by replacing the above errors with the \emph{$m$-the Chebyshev greedy error}, defined as 
\begin{align*}
\vartheta_m\left(f\right):=\sup_{A\in GS\left(f,m\right)}\inf_{g\in \spn\lbrace{\xx_n: n\in A}\rbrace}\left\Vert f-g\right\Vert_{\XX}.
\end{align*}
Finally, in all cases, we will extend the results from the classical weights $\left(k^{\alpha}\right)_{k\in\NN}$ to wider classes of weights. \\
The structure of the paper is as follows: In Section~\ref{sectionsetting}, we outline the main setting of our study, and set up our notation. In Section~\ref{sectionlorentz}, we construct embeddings from approximation classes into weighted Lorentz sequence spaces and vice versa, extending some  known results for unconditional bases to wider classes of bases. 
In Section~\ref{sectiongreedy=classical}, we study sufficient conditions for equality between classical approximation spaces, greedy approximation classes and Chevyshev-greedy approximations classes, and give  necessary and sufficient conditions in the case of truncation quasi-greedy bases with a mild convergence property.

\section{General setting and notation}\label{sectionsetting}

Throughout this paper, we will assume that $\XX$ is a separable and infinite dimensional quasi-Banach (or $p$-Banach space) over the field $\mathbb{F} = \mathbb{R}$ or $\mathbb{C}$. Let $\mathbb{X}^*$ be the dual space of $\mathbb{X}$. We call a sequence $\XB = (\xx_n)_{n\in \NN}\subset \mathbb{X}$ a \textbf{basis} of $\mathbb{X}$ if 
\begin{enumerate}[\rm (1)]
    \item $\mathbb{X} = \overline{\langle \XB\rangle}$, where $\langle \XB\rangle$ represents the linear span of $\XB$;
    \item there is a sequence $\XB^*=(\xx_n^*)_{n\in \NN}\subset \mathbb{X}^*$ such that $\xx_j^*(\xx_k) = \delta_{j, k}$ for all $j, k\in\mathbb{N}$.
      \end{enumerate}
We will call $\XB^*$ the \emph{dual basis} of $\XB$. \\
If $\XB$ also satisfies 
\begin{enumerate}
\item[ (3)] $\mathbb{X}^*\ =\ \overline{\langle  \xx_n^*: n\in \mathbb{N}\rangle}^{w^*},$
\end{enumerate}
then $\XB^*$ is \emph{total} and $\XB$ is a \textit{Markushevich basis}. Additionally, if the partial sum operators $S_m[\XB,\XX]\left(f\right)=S_m(f):= \sum_{n=1}^m \xx_n^*\left(f\right)\xx_n$ for $m\in \mathbb{N}$ are uniformly bounded, i.e., there exists $ C > 0$ such that 
\begin{enumerate}
\item [(4)] $\left\Vert S_m\left(f\right)\right\Vert_{\XX}\ \le C\left\Vert f\right\Vert_{\XX},\; \forall f\in \mathbb{X}, \forall m\in \mathbb{N}$,
\end{enumerate}
then we say $\XB$ is a \textit{Schauder basis}, and we say that $\XB$ is \textit{unconditional} if there is $C>0$ such that 
\begin{enumerate}
\item[(5)] $\left\Vert P_A\left(f\right)\right\Vert_{\XX}\le C\left\Vert f\right\Vert_\XX,\; \forall f\in \mathbb{X}, \forall \vert A\vert<\infty,$
\end{enumerate} 
where $\vert A\vert$ is the cardinality of the set of indices $A$ and $P_A$ is the \textit{projection operator}, that is, given a finite set $A$ of indices, $P_A[\XB,\XX](f)=P_A(f):=\sum_{n\in A}\xx_n^*\left(f\right)\xx_n$. 
From now on, unless otherwise stated we  assume that the quasi-norm on $\XX$ is a  continuous $p$-norm for some $0<p\le 1$ and that all bases are bounded with bounded dual bases. We set
\begin{align*}
&\alpha_1=\alpha_1[\XB,\XX, \left\Vert \cdot\right\Vert_{\XX}]:=\sup_{n\in \NN}\left\Vert \xx_n\right\Vert_{\XX};\\ &\alpha_2=\alpha_2[\XB,\XX, \left\Vert \cdot\right\Vert_{\XX}]:=\sup_{n\in \NN}\left\Vert \xx_n^*\right\Vert_{\XX};\\
&\alpha_3=\alpha_3[\XB,\XX, \left\Vert \cdot\right\Vert_{\XX}]:=\sup_{n\in \NN}\left\Vert \xx_n\right\Vert_{\XX}\left\Vert \xx_n^*\right\Vert_{\XX},
\end{align*}
and we denote $\alpha\left(\XX\right)$ the  modulus of concavity of $\XX$. 

Also, when we have a Markushevich basis, we can define, for each $f\in\XX$, the support:

$$\supp[\XB,\XX](f)=\supp(f):=\lbrace n\in\NN : \xx_n^*(f)\neq 0\rbrace.$$
In connection with the support, is very usual to take elements in the cube $\mathcal Q$, where
$$\mathcal Q[\XB,\XX]=\mathcal Q:=\lbrace f\in\mathbb X : \vert \xx_n^*(f)\vert\leq 1,\; \forall n\in\mathbb N\rbrace.$$

Some notation about the bases in quasi-Banach spaces that we need is the following:
\begin{itemize}
	\item For each $A\in\NN^{<\infty}$, we set $\XB_{A^c}:=\left(\xx_n\right)_{n\not\in A}$. 
	\item Given $\BB\subset \XB$, by $SG\left(\BB\right)$ we denote the set of elements $f\in \overline{\langle \BB\rangle}$ such that $\left\vert \xx_n^*\left(f\right)\right\vert\not=\left\vert \xx_k^*\left(f\right)\right\vert$ for each $n,k\in \supp\left(f\right)$ with $k\not=n$. 
\end{itemize}

\subsection{Sets.} In this paper, we will use set of indices $A$ of $\mathbb N$. For that reason, we need the following notation: 
\begin{itemize}
\item $\NN^{<\infty}$ denotes the set of finite subsets of $\NN$. 
\item For each $m\in \NN_0$, by $\NN^{(m)}$ we denote the set of subsets of $\NN$ of cardinality $m$, whereas by  $\NN^{\le m}$ we denote the set of subsets of $\NN$ of cardinality no greater than $m$. 
\item For each $k\in \NN$, $k\NN$ denotes the set of multiples of $k$. 
\item Given $A\subset \RR$ and $a\in \RR$, we note $A_{>a}:=\left\lbrace x\in A: x>a\right\rbrace$. We use similar notation for $\le$, $\ge$ and $>$.  
\item If $A, B\subset \RR$, we note $A<B$ if $a<b$ for all $a\in A$, $b\in B$. Also, by $m<B$ we mean $\{m\}<B$. We use similar notation for $\le$, $\ge$ and $>$.
\end{itemize}

\subsection{Indicator sums and democracy-like properties.} Let $\XX$ be a quasi-Banach space with $\XB=(\xx_n)_{n\in\mathbb N}$ a Markushevich basis in $\XX$. One of the most powerful tools that we will use here and that we will need to show our most important results involving Lorentz spaces, are the democracy functions, which are defined through the use of indicator sums. For that, for each $f\in \XX$, $\varepsilon\left(f\right)\in \EE^{\NN}$ is the sequence of signs $\left(\sgn\left(\xx_n^*\left(f\right)\right)\right)_{n\in \NN}$, 
where for every $z\in \FF$, 
\begin{align*}
	\sgn\left(z\right):=&
	\begin{cases}
		\frac{z}{\left\vert z\right\vert} & \text{ if }z\not=0;\\
		1 &\text{ if }z=0,
	\end{cases}
\end{align*}
and $\EE:=\left\lbrace t\in \FF: \left\vert t\right\vert=1\right\rbrace$. 

Given $A\in \NN^{<\infty}$ and $\varepsilon\in \EE^{\NN}$, we define the \textit{indicators sums} like
\begin{align*}
	&\Ind_{\varepsilon, A}=\Ind_{\varepsilon, A}\left[\XB,\XX\right]:=\sum_{n\in A}\varepsilon_n\xx_n &&\text{and}&&&\Ind_{A}=\Ind_{A}\left[\XB,\XX\right]:=\sum_{n\in A}\xx_n,
\end{align*}

and \textit{the left and right superdemocracy functions of $\XB$} (cf. \cite{GHN2012}) are defined as follows: for each $n\in \NN$, 
\begin{align*}
	&\hl\left(n\right):=\inf_{\substack{A\in \NN^{(n)}\\\varepsilon\in \EE^{\NN}}}\left\Vert \Ind_{\varepsilon, A}\right\Vert_{\XX},&&\hr\left(n\right):=\sup_{\substack{A\in \NN^{(n)}\\\varepsilon\in \EE^{\NN}}}\left\Vert \Ind_{\varepsilon, A}\right\Vert_{\XX}.
\end{align*}

Using the indicator sums, we can define some \textit{democracy-like properties} of the basis $\XB$, that are crucial notions to characterize some greedy-like bases as we will see in the next subsection.
\begin{definition}
	We say that $\XB$ is $C$-superdemocratic for some $C>0$ if
	\begin{align}
		\left\Vert \Ind_{\varepsilon,A}\right\Vert\le& \left\Vert \Ind_{\varepsilon',B}\right\Vert\label{superdemdefi}
	\end{align}
	for all $A, B\in \NN^{<\infty}$ with $\left\vert A\right\vert\le \left\vert B\right\vert$ and $\varepsilon, \varepsilon'\in \EE^{\NN}$. If \eqref{superdemdefi} holds at least when $\varepsilon_n=\varepsilon'_n=1$ for all $n\in \NN$, we say that $\XB$ is $C$-democratic. 
\end{definition}

\begin{definition}
	We say that $\XB$ is $C$-suppression unconditional for constant coefficients ($C$-SUCC) for some $C>0$ if 
	\begin{align}
		\left\Vert \Ind_{\varepsilon,A}\right\Vert\le& C\left\Vert \Ind_{\varepsilon,B}\right\Vert\nonumber
	\end{align}
	for all $A, B\in \NN^{<\infty}$ with $A\subset B$ and all $\varepsilon\in \EE^{\NN}$. 
\end{definition}

\begin{definition}
	We say that $\XB$ is $C$-symmetric for largest coefficients for some $C>0$ if 
	\begin{align*}
		\left\Vert \Ind_{\varepsilon, A}+f\right\Vert_{\XX}\le& C\left\Vert \Ind_{\varepsilon', B}+f\right\Vert_{\XX}
	\end{align*}
	for all $A, B\in \NN^{<\infty}$ with $|A|\le |B|$, all $\varepsilon, \varepsilon'\in \EE^{\NN}$ and all $f\in \Cu$ with $\supp\left(f\right)\cap \left(A\cup B\right)=\emptyset$.
\end{definition}

\subsection{Thresholding Greedy Algorithm.} As we have commented in the introduction, one of the main techniques to study the connection between approximation spaces and Lorentz spaces is the use of the Thresholding Greedy Algorithm (TGA). The TGA was introduced in \cite{KT1999} and, basically, the algorithm selects the largest coefficients (in modulus) of $f$. This idea produces the definition of \textit{greedy set}: if $\XB$ is a Markushevich basis in a quasi-Banach (or $p$-Banach) space, for each $f\in \XX$ and $m\in \NN_0$, by $GS\left(f,m \right)$ we denote the set of greedy sets of $f$ of cardinality $m$, that is, sets $A\subset \NN$ with $\left\vert A\right\vert =m$ such that
\begin{align}
	&\left\vert \xx_n^*\left(f\right)\right\vert\ge \left\vert \xx_k^*\left(f\right)\right\vert&&\forall n\in A\forall k\not\in A,\label{greedyset}
\end{align}
and by $SGS\left(f,m \right)$ we denote the set of strictly greedy sets of $x$ of cardinality $m$, that is sets such that strict inequality holds in \eqref{greedyset}.  For each $f\in \XX$,  we will also use the notation $NSG\left(f\right)$ denoting the set of natural numbers $m$ for which there is a unique (strictly) greedy set of $f$ of cardinality $m$. 

One possible way to construct greedy sets is the use of \textit{greedy orderings}, where given $f\in \XX$, a greedy ordering for $f$ is an inyective funcion $\rho: \NN\rightarrow \NN$ such that $\supp\left(f\right)\subset \rho\left(\NN\right)$ and, for each $n\in \NN$, 
\begin{align*}
	\left\lbrace \rho\left(k\right): 1\le k\le n\right\rbrace \in GS\left(f,n\right). 
\end{align*}

Then, a \textit{greedy sum of order $m$ of $f$} is the sum
$$G_m(f):=\sum_{j=1}^m\xx_{\rho(j)}^*(f)\xx_{\rho(j)}),$$
and the collection $(G_m)_{m\in\mathbb N}$ is the TGA.

Regarding the convergence of the TGA, we can find different types of greedy-like bases, which are of use to study the convergence and general behavior of the TGA. 

\begin{definition}
We say that $\XB$ is $C$-quasi-greedy if
\begin{align*}
	\left\Vert P_A\left(f\right)\right\Vert_{\XX}\le C\left\Vert f\right\Vert_{\XX}
\end{align*}
for all $f\in \XX$, $m\in \NN$, and $A\in GS\left(f,m\right)$, and is $C$-suppression quasi-greedy if
\begin{align*}
	\left\Vert f- P_A\left(f\right)\right\Vert_{\XX}\le C\left\Vert f\right\Vert_{\XX}
\end{align*}
under the same conditions as above.  
\end{definition}
The connection between this definition and the convergence of the algorithm was given by P. Wojtaszczyk in \cite{W2000} where the author proved that a basis is quasi-greedy if and only if 
\begin{eqnarray}\label{convergence}
	\lim_{n\rightarrow+\infty}\left\Vert f-\sum_{j=1}^n \xx_{\rho(j)}^*(f)\xx_{\rho(j)}\right\Vert=0,\; \forall f\in\XX.
\end{eqnarray}

Some of our results are in terms of a property weaker than quasi-greediness \cite[Lemma 2.2]{DKKT2003}, \cite[Theorem 4.13]{AABW2021}, defined as follows:

\begin{definition}\label{definitionTQG}
	We say that $\XB$ is $C$-truncation quasi-greedy if 
	\begin{align*}
		\min_{n\in A}\left\vert\xx_n^*\left(f\right)\right\vert\left\Vert \Ind_{\varepsilon\left(f\right),A}\right\Vert_{\XX}\le C\left\Vert f\right\Vert_{\XX} 
	\end{align*}
	for all $f\in \XX$,  $m\in \NN$, and $A\in GS\left(f,m\right)$. 
\end{definition}

It will also be convenient to use an equivalent property \cite[Theorem 1.5]{AAB2024}. 

\begin{definition}\label{definitionBOU}\cite{DOSZ2009}
We say that $\XB$ is $(1,1)$-bounded-oscillation unconditional with constant $C$ (or $C$-BOU) if 
\begin{align*}
	\left\Vert \Ind_{\varepsilon, A}\right\Vert_{\XX}\le& C\left\Vert \Ind_{\varepsilon, A}+f\right\Vert_{\XX}
\end{align*}
for all $A\in \NN^{<\infty}$, all $\varepsilon\in \EE^{\NN}$ and all $f\in \XX$ with $\supp\left(f\right)\cap A=\emptyset$. 
\end{definition}


As we can see in \eqref{convergence}, quasi-greediness is the minimal condition that guarantees the convergence of the TGA. Two stronger conditions - which give better approximations - are greedy bases and almost greedy bases, introduced in \cite{KT1999} and \cite{DKKT2003} respectively. Greedy and almost greedy bases are respectively defined in terms of the best $m$-term approximation error and the best $m$-term approximation error by projections:
$$\sigma_m[\XB,\XX](f)=\sigma_m(f):=\inf\left\lbrace \left\Vert f-\sum_{n\in B}a_n \xx_n\right\Vert_\XX : a_n\in\mathbb F, \vert B\vert\leq m\right\rbrace,$$
and
$$\widetilde{\sigma}_m[\mathcal B,\XX](f)=\widetilde{\sigma}_m(f):=\inf_{\vert A\vert\leq m}\Vert f-P_A(f)\Vert_\XX.$$
\begin{definition}
We say that $\XB$ is $C$-greedy for some $C>0$ if
$$\Vert f-P_A(f)\Vert_\XX \leq C \sigma_{\vert A\vert}(f),$$
for all $f\in\mathbb X$ and all greedy set $A$. Also, we say that $\XB$ is $C$-almost greedy if
$$\Vert f-P_A(f)\Vert_\XX\leq C\widetilde{\sigma}_{\vert A\vert}(f),$$
for all $f\in\mathbb X$ and all greedy set $A$.
\end{definition}

The main characterization of greediness was given in the context of Banach spaces in \cite{KT1999} and, in \cite{AABW2021}, the authors extended the result to quasi-Banach spaces.

\begin{theorem}
Let $\XX$ be a quasi-Banach space and $\mathcal B$ a Markushevich basis of $\XX$. The following are equivalent:
\begin{itemize}
	\item $\XB$ is greedy.
	\item $\XB$ is unconditional and democratic.
	\item $\XB$ is unconditional and superdemocratic.
	\item $\XB$ is unconditional and symmetric for largest coefficients.
\end{itemize}
\end{theorem}

There are similar characterizations of almost greediness, proved in \cite{DKKT2003} for Banach spaces and extended in \cite{AABW2021} to quasi-Banach ones. 

\begin{theorem}
	Let $\XX$ be a quasi-Banach space and $\mathcal B$ a Markushevich basis of $\XX$. The following are equivalent:
	\begin{itemize}
		\item $\XB$ is almost greedy.
		\item $\XB$ is quasi-greedy and democratic.
		\item $\XB$ is quasi-greedy and superdemocratic.
		\item $\XB$ is quasi-greedy and symmetric for largest coefficients.
	\end{itemize}
\end{theorem}

\subsection{Weights.} 
For some of our results, we need to define different types of weights; more precisely, we will study general approximation spaces $\mathcal A_q^w$ and classes $\GG^w_q$ and $\CG^w_q$, where \emph{weight} $w=\left(w\left(n\right)\right)_{n\in\NN}$ is a \emph{weight}, that is a sequence of positive numbers; we define $\WW$ as the set of all weights. Given a weight $w$, $\widetilde{w}$ is the weight given by $\widetilde{w}\left(n\right)=\sum_{j=1}^{n}\frac{w\left(j\right)}{j}$. Some classes of weights that we need are the following ones:
\begin{itemize}
	\item $\WW_{i}$ is the set of non-decreasing weights. 
	\item $\WW_d$ is the set of doubling weights, weights for which $w(2n)\le Cw(n)$ for some fixed $C>0$.  
	\item $\WW_u$ is the set of unbounded weights. 
\end{itemize}

For each $w\in \WW_{i}$, we associate the following functions from $\NN$ to $\RR$: 
	\begin{align*}
		&\varphi_w\left(m\right):=\inf_{k\in \NN}\frac{w\left(km\right)}{w\left(k\right)},&&\Phi_w\left(m\right):=
		\sup_{k\in \NN}\frac{w\left(km\right)}{w\left(k\right)}.
	\end{align*}
	Also, we define the parameters

	\begin{align*}
		&i_w:=\sup_{m>1}\frac{\log\left(\varphi_w\left(m\right)\right)}{\log\left(m\right)};&&I_w:=\inf_{m>1}\frac{\log\left(\Phi_w\left(m\right)\right)}{\log\left(m\right)}.
	\end{align*}
\begin{itemize}
	\item $\WW_{+}$ is the set of weights in $\WW_{i}$ such that $i_w>0$.
	\item $\WW_{-}$ is the set of weights in $\WW_{i}$ such that $I_w<1$
	\item We combine the indices of the weights defined above: for example, $\WW_{i,d}$ is the set of non-decreasing, doubling weights, and so on. 
\end{itemize}
\begin{definition}
A weight $w\in \WW$ has  the Lower Regularity Property (LRP) if there are $C>0$ and $0<\beta<1$ such that
	\begin{align*}
		&C\left(\frac{w\left(m\right)}{w\left(n\right)}\right)\ge \left(\frac{m}{n}\right)^{\beta}&&\forall 1\le n\le m. 
	\end{align*}
	In that case, we say that $w$ has the $\beta$-LRP. Moreover, a weight $w\in \WW$ has  the Upper Regularity Property (URP) if there are $C>0$ and $0<\beta<1$ such that
	\begin{align*}
		&C\left(\frac{w\left(m\right)}{w\left(n\right)}\right)\le \left(\frac{m}{n}\right)^{\beta}&&\forall 1\le n\le m. 
	\end{align*}
	In that case, we say that $w$ has the $\beta$-URP. 
\end{definition}
We will use the following result from the literature. 

\begin{theorem}\cite[Theorem 8.5]{BCH2023}\label{theoremLRPURP} Let $w \in \WW_{i,d}$. Then
	\begin{enumerate}[\rm (i)]
		\item $w$ has the LRP if and only if $i_w>0$. Moreover
		\begin{align}\label{lrp1}i_w=\sup\left\lbrace \alpha>0: w \in LRP\left(\alpha\right)\right\rbrace.
		\end{align}
		\item $w$ has the URP if and only if $I_w<1$. Moreover
		\begin{align}\label{urp1}I_w=\inf\left\lbrace \beta<1: w \in URP\left(\beta\right)\right\rbrace. 
		\end{align}
	\end{enumerate}
\end{theorem}

\subsection{Greedy and approximation classes.}
Given a basis $\XB$ of a $p$-Banach space $\XX$, the best-approximation classes are defined as follows: 
\begin{align*}
&{\A}^{w}_q:=\left\lbrace f\in \XX: \left\Vert f\right\Vert_{\A^{w}_q}:=\left\Vert f\right\Vert_{\XX}+\left(\sum_{n\in\NN}\left(w(n)\sigma_n(f)\right)^q\frac{1}{n}\right)^{\frac{1}{q}}<\infty\right\rbrace\\
& \text{ if }0<q<\infty;\\
&{\A}^{w}_{\infty}=\left\lbrace f\in \XX: \left\Vert f\right\Vert_{\A^{w}_{\infty}}:=\left\Vert f\right\Vert_{\XX}+\sup_{n\in\NN} w(n)\sigma_n(f) <\infty\right\rbrace. 
\end{align*} 
\begin{remark}\rm \label{remarkQB}It is easy to check  that if $w\in \WW_{i,d}$, then $\|\cdot\|_{\A^{w}_q}$ and $\A^w_q$ is a quasi-Banach space. 
\end{remark}
We will also consider the  error of greedy approximation (see \cite{GHN2012}) given by 
\begin{align*}
&\gamma_n\left(f\right):=\sup_{A\in GS\left(f,n\right)}\left\Vert f-P_A\left(f\right)\right\Vert_{\XX}&&\forall n\in \NN_0,
\end{align*}
and the greedy classes, defined as 
\begin{align*}
&{\GG}^{w}_q:=\left\lbrace f\in \XX: \left\Vert f\right\Vert_{\GG^{w}_q}:=\left\Vert f\right\Vert_{\XX}+\left(\sum_{n\in\NN}\left(w(n)\gamma_n(f)\right)^q\frac{1}{n}\right)^{\frac{1}{q}}<\infty\right\rbrace\\
& \text{ if }0<q<\infty;\\
&{\GG}^{w}_{\infty}=\left\lbrace f\in \XX: \left\Vert f\right\Vert_{\GG^{w}_{\infty}}:=\left\Vert f\right\Vert_{\XX}+\sup_{n\in\NN} w(n)\gamma_n(f) <\infty\right\rbrace. 
\end{align*} 
Similarly, the error of Chevyshev-greedy approximation is 
\begin{align*}
&\vartheta_n\left(f\right):=\sup_{A\in GS\left(f,n\right)}\inf_{g\in \langle \xx_n: n\in A\rangle}\left\Vert f-g\right\Vert_{\XX}&&\forall n\in \NN_0,
\end{align*}
and the Chevyshev-greedy classes (see \cite{BCH2023}) defined as 
\begin{align*}
&{\CG}^{w}_q:=\left\lbrace f\in \XX: \left\Vert f\right\Vert_{\CG^{w}_q}:=\left\Vert f\right\Vert_{\XX}+\left(\sum_{n\in\NN}\left(w(n)\vartheta_n(f)\right)^q\frac{1}{n}\right)^{\frac{1}{q}}<\infty\right\rbrace\\
& \text{ if }0<q<\infty;\\
&{\CG}^{w}_{\infty}=\left\lbrace f\in \XX: \left\Vert f\right\Vert_{\CG^{w}_{\infty}}:=\left\Vert f\right\Vert_{\XX}+\sup_{n\in\NN} w(n)\vartheta_n(f) <\infty\right\rbrace. 
\end{align*} 

We will write $\GG^{w}_q\hookrightarrow \A^{w}_q$ there is $C>0$ such that $\left\Vert \cdot\right\Vert_{\A^{w}_q}\le C\left\Vert \cdot\right\Vert_{\GG^{w}_q}$. On the other hand, we will use $\GG^{w}_q= \A^{w}_q$ to mean set equality, whereas $\GG^{w}_q\approx \A^{w}_q$ means that $\GG^{w}_q\hookrightarrow \A^{w}_q$ and $\A^{w}_q\hookrightarrow \GG^{w}_q$. We will use the symbols $\hookrightarrow$ and $\approx$ in the same manner for other spaces.

\begin{remark}\label{remarkonlyonegreedy} \rm
Note that 
\begin{itemize}
\item For each $f\in \XX$ and each $0<q\le \infty$, 
\begin{align}
&\left\Vert f\right\Vert_{\A^{w}_q}\le  \left\Vert f\right\Vert_{\CG^{w}_q}\le \left\Vert f\right\Vert_{\GG^{w}_q},\label{keyrel1}
\end{align}
with the convention that they may take the value $\infty$. 
\item If $f\in SG\left(\XB\right)$, for each $j\in \NN$ and $A\in GS\left(f,j\right)$ we have $\left\Vert f-P_A\left(f\right)\right\Vert=\gamma_j\left(f\right)$ (in fact, $GS\left(f,j\right)=SGS\left(f,j\right)=\left\lbrace A\right\rbrace$ provided that $A\subset \supp\left(f\right)$).  
\item If $f \in \langle \XB\rangle$, then $\gamma_n\left(f\right)=0$ for each $n\ge \left\vert \supp\left(f\right)\right\vert$.
\item $\XB$ is also seminormalized with respect to $\GG^{w}_q$. Thus, in the cases where $\left\Vert\cdot\right\Vert_{\GG^{w}_q}$ or $\left\Vert\cdot\right\Vert_{\CG^{w}_q}$  is a quasi-norm, $\XB^*$ is seminormalized in the respective dual space. 
\end{itemize}
\end{remark}

\section{Approximation classes and Lorentz spaces}\label{sectionlorentz}
In this section, we study the relation between approximation classes and weighted Lorentz sequence spaces, and show that there are embeddings under weaker conditions than known thus far. Given a weight $\eta\in \WW$ and $0<q<\infty$, the weighted Lorentz space $\ell^{\eta}_q$ is defined by
\begin{align*}
\ell_{\eta}^q:=&\left\lbrace \textbf{s}=\left(s_n\right)_{n\in\NN}\in \mathtt{c}_0: \left\Vert \textbf{s}\right\Vert_{\ell_{\eta}^q}<\infty\right\rbrace, 
\end{align*}
where $\ell^{\eta}_q$ is a quasi-norm defined as  
\begin{align*}
\left\Vert \textbf{s}\right\Vert_{\ell_{\eta}^q}=\left(\sum_{n=1}^{\infty}\left(\eta\left(n\right)s_n^*\right)^{q}\frac{1}{n}\right)^{\frac{1}{q}},
\end{align*}
and $\left(s_n^*\right)_{j\in \NN}$ is the non-increasing rearrangement of $\textbf{s}$. For $q=\infty$, the definition is
\begin{align*}
\left\Vert \textbf{s}\right\Vert_{\ell_{\eta}^q}=&\sup_{n\in \NN}\eta\left(n\right)s_n^*. 
\end{align*}
\begin{remark}\label{remarkchequear} It is well known that if $\eta\in \WW_{i,d,u}$, the space $\ell_\eta^q$ is a quasi-Banach space (see for instance \cite[Subsection 2.2]{CRS}). Equivalent quasi-norms are given by
$$\Vert  \textbf{s}\Vert_{\ell_\eta^q}\approx \left(\sum_{j=0}^\infty (\eta(\kappa^j)s_{\kappa^j}^*)^q\right)^{\frac{1}{q}}.$$
\end{remark}
Our next two results involve Bernstein-type inequalities (see \cite{GHN2012}), and improve some of the results of \cite[Theorem 4.2]{GHN2012} by replacing the unconditionality hypothesis with truncation quasi-greediness. To simplify the notation, given a basis $\XB$ of $\XX$ and $f\in \XX$, $\eta\in \WW$ and $0<q\le \infty$, we define 
\begin{align*}
\left\Vert f\right\Vert_{\ell^{q}_{\eta}}:=&\left\Vert \left(\xx_n^*\left(f\right)\right)_{n\in \NN} \right\Vert_{\ell^{q}_{\eta}},
\end{align*}
where the value $\infty$ is allowed. 
\begin{lemma}\label{lemmaBernstein1}Let $\XB$ be a truncation quasi-greedy basis of $\XX$, and $\eta,  w \in \WW_{i,d}$, and suppose there is $C>0$ such that 

\begin{align*}
&C\hl\left(n\right)\ge \eta\left(n\right)&&\forall n\in \NN. 
\end{align*}

Then for each $0< q\le \infty$, there is $C_q>0$ such that 
\begin{align*}
&\left\Vert f\right\Vert_{\ell^{q}_{\eta w}}\le C_q \left\Vert f\right\Vert_{\A^{w}_q}&&\forall f\in \A^{w}_q. 
\end{align*}
\end{lemma}
\begin{proof}
Suppose that $w$ is $C_0$-doubling, $\eta$ is $C_1$-doubling, and $\XB$ is $C_2$-truncation quasi-greedy. Given $f\in \A^{w}_q$ and $n\in \NN$, fix $\epsilon>0$ and pick $g\in \langle \XB\rangle$ with $\left\vert \supp\left(g\right)\right\vert\le n$ so that 
\begin{align*}
\left\Vert f-g\right\Vert_{\XX}\le& \sigma_n\left(f\right)+\epsilon.
\end{align*}
Now let $A\in GS\left(f-g,n\right)$. Note that 
\begin{align*}
\left\vert\left\lbrace k\in \NN: \left\vert \xx_k^*\left(f-g\right)\right\vert \ge f_{2n}^* \right\rbrace\right\vert\ge n, 
\end{align*}
so 
\begin{align*}
\left(f-g\right)_{n}^*=\min_{k\in A}\left\vert \xx_k^*\left(f-g\right)\right\vert\ge f_{2n}^*. 
\end{align*}
Hence, 
\begin{align*}
\eta\left(2n\right)f_{2n}^*\le& C_1\eta\left(n\right)f_{2n}^*\le C_1C\left(f-g\right)_{n}^* \left\Vert \Ind_{\varepsilon\left(f-g\right), A}\right\Vert_{\XX}\\
\le& CC_1C_2\left\Vert f-g\right\Vert_{\XX}\le CC_1C_2 \sigma_n\left(f\right)+CC_1C_2\epsilon. 
\end{align*}
Since $\epsilon$ is arbitrary, it follows that 
\begin{align*}
w\left(2n\right)\eta\left(2n\right)f_{2n}^*\le&  CC_0C_1C_2 w\left(n\right) \sigma_n\left(f\right).
\end{align*}
Since $$w\left(2n+1\right) \eta\left(2n+1\right)f_{2n+1}^*\le C_0C_1w\left(2n\right) \eta\left(2n\right)f_{2n}^*,$$
in the case $q=\infty$ we have
\begin{align*}
&\left\Vert f\right\Vert_{\ell^{q}_{\eta w}}= \max\left\lbrace w\left(1\right)\eta\left(1\right)w\left(1\right)f_1^*,\sup_{n\in \NN_{\ge 2}} w\left(n\right)\eta\left(n\right)f_{n}^*\right\rbrace \\
\le& \eta\left(1\right)w\left(1\right)\alpha_2\left\Vert f\right\Vert_{\XX}+CC_0^2C_1^2C_2\sup_{n\in\NN}w\left(n\right)\sigma_n\left(f\right)\\
\le& \max\left\lbrace CC_0^2C_1^2C_2,\eta\left(1\right)w\left(1\right)\alpha_2\right\rbrace\left\Vert f\right\Vert_{\A^{w}_{\infty}},   
\end{align*}
whereas if $0<q<\infty$, a similar computation gives
\begin{align*}
&\left\Vert f\right\Vert_{\ell^{q}_{\eta w}}=\left(\sum_{n=1}^{\infty}\left(w\left(n\right)\eta\left(n\right)f_n^*\right)^{q}\frac{1}{n}\right)^{\frac{1}{q}}\\
&\le 2^{\frac{1}{q}}w\left(1\right)\eta\left(1\right)w\left(1\right)f_1^*+2^{\frac{1}{q}}\left(\sum_{n=2}^{\infty}\left(w\left(n\right)\eta\left(n\right)f_n^*\right)^{q}\frac{1}{n}\right)^{\frac{1}{q}}\\
& 
\le \max\left\lbrace 2^{\frac{1}{q}+1}CC_0^2C_1^2C_2,2^{\frac{1}{q}}\eta\left(1\right)w\left(1\right)\alpha_2\right\rbrace\left\Vert f\right\Vert_{\A^{w}_{q}}.
\end{align*}
\end{proof}
A result in the opposite direction can be obtained under different conditions on the weights and the basis.  To prove it,  we will use Theorem~\ref{theoremLRPURP}.

\begin{lemma}\label{lemmabernstein2}Let $\XB$ be basis of $\XX$, and let $\eta\in \WW_{i, d}$ and $w \in \WW_{i,d,+}$. Suppose that $\XB$ is suppression unconditional for constant coefficients, and there exist $C>0$ and $0<q\le \infty$ such that
\begin{align*}
&\left\Vert \Ind_{\varepsilon, A}\right\Vert_{\ell^{q}_{\eta w}}\le C \left\Vert \Ind_{\varepsilon, A}\right\Vert_{\GG^{w}_q}&&\forall A\in \NN^{<\infty}\forall \varepsilon\in \EE^A. 
\end{align*}
Then there is $C'>0$ such that 
\begin{align*}
&\hl\left(n\right)\ge C'\eta\left(n\right)&&\forall n\in \NN. 
\end{align*}
\end{lemma}
\begin{proof}
By Theorem~\ref{theoremLRPURP}, $w$ has the LRP. Since $\eta$ is non-decreasing, it follows that $\eta w$ has the LRP. Hence, again by Theorem~\ref{theoremLRPURP}, we have $\eta w\in \WW_{i,d,+}$. Suppose now that $0<q<\infty$. By \cite[Lemma 2.3]{GHN2012} there is $C_1>0$ such that
\begin{align*}
&\frac{1}{C_1}\eta\left(n\right)w\left(n\right)\le   \left\Vert\Ind_{\varepsilon, A}\right\Vert_{\ell^{\eta w}_q}\le C_1\eta\left(n\right)w\left(n\right)&&\forall n\in \NN, A\in \NN^{(n)}, \varepsilon\in \EE^{\NN}.
\end{align*}
Suppose $\XB$ is $C_2$-SUCC. Since $w^q\in \WW_{i,d,+}$, by \cite[Propositions 2.4,2.5]{BBGHO2018} there is $C_3\ge 1$ such that
\begin{align*}
&\frac{1}{C_3}\widetilde{w^q}(n) \le  w^q(n)\le C_3 \widetilde{w^q}(n) &&\forall n\in \NN. 
\end{align*}
Now fix $n\in \NN$, choose $\epsilon>0$ and pick $A\in \NN^{(n)}$ and $\varepsilon\in \EE^{\NN}$ so that 
\begin{align*}
\left\Vert\Ind_{\varepsilon, A}\right\Vert_{\XX}\le & \hl\left(n\right)+\epsilon. 
\end{align*}
Note that for each $k\in \NN$, 
\begin{align*}
\gamma_k\left(\Ind_{\varepsilon, A}\right)\le& \max_{B\subset A}\left\Vert \Ind_{\varepsilon, B}\right\Vert_{\XX}\le C_2 \left\Vert \Ind_{\varepsilon, A}\right\Vert_{\XX}.
\end{align*}
Thus, 
\begin{align*}
\left\Vert \Ind_{\varepsilon, A}\right\Vert_{\GG^{w}_q}=&\left\Vert\Ind_{\varepsilon, A}\right\Vert_{\XX}+\left(\sum_{k=1}^{n}\left(w\left(k\right)\gamma_k\left(\Ind_{\varepsilon, A}\right)\right)^{q}\frac{1}{k}\right)^{\frac{1}{q}}\\
\le& \left\Vert\Ind_{\varepsilon, A}\right\Vert_{\XX}\left(1+C_2 \left(\widetilde{w^q}\left(n\right)\right)^{\frac{1}{q}}\right)\le \left\Vert\Ind_{\varepsilon, A}\right\Vert_{\XX} \left(1+C_2C_3^{\frac{1}{q}} w\left(n\right)\right)\\
\le& \left(\frac{1}{w\left(1\right)}+
 C_2C_3^{\frac{1}{q}}\right)w\left(n\right)\hl\left(n\right)+\left(1+C_2C_3^{\frac{1}{q}} w\left(n\right)\right)\epsilon. 
\end{align*}
Since $\epsilon$ is arbitrary, combining the above inequalities we deduce that 
\begin{align*}
\eta\left(n\right)\le& CC_1\left(\frac{1}{w\left(1\right)}+
 C_2C_3^{\frac{1}{q}}\right)\hl\left(n\right). 
\end{align*}
Now suppose $q=\infty$. By \cite[Lemma 2.3]{GHN2012}, we have
\begin{align*}
& \left\Vert\Ind_{\varepsilon, A}\right\Vert_{\ell^{\eta w}_{\infty}}=\eta\left(n\right)w\left(n\right)&&\forall n\in \NN, A\in \NN^{(n)}, \varepsilon\in \EE^{\NN}.
\end{align*}
Let $C_2$ be as above, fix $n\in \NN$ and pick $A, \varepsilon$ as before. We have
\begin{align*}
\left\Vert\Ind_{\varepsilon, A}\right\Vert_{\GG^{w}_{\infty}}=&\left\Vert\Ind_{\varepsilon, A}\right\Vert_{\XX}+\max_{1\le k\le n}w\left(k\right)\gamma_k\left(\Ind_{\varepsilon, A}\right)\\\le& \left\Vert\Ind_{\varepsilon, A}\right\Vert_{\XX}+C_2 w\left(n\right)\left\Vert\Ind_{\varepsilon, A}\right\Vert_{\XX}\\
\le& \left(\frac{1}{w\left(1\right)}+C_2\right) w\left(n\right)\hl\left(n\right)+ \left(\frac{1}{w\left(1\right)}+C_2\right)w\left(n\right)\epsilon. 
\end{align*}
Combining we get
\begin{align*}
\eta\left(n\right)\le& C \left(\frac{1}{w\left(1\right)}+C_2\right) \hl\left(n\right).
\end{align*}
\end{proof}
Next, we study Jackson-type inequalities (see \cite{GHN2012}). We will use the following definition. 
\begin{definition}\label{definitionproyections}Let $\XB$ be a basis of $\XX$. We say that $f\in \XX$ is \emph{approximable by proyections} if there is a sequence of sets $\left(B_n\right)_{n\in \NN}\subset\NN^{<\infty}$ such that
\begin{align*}
&\lim_{n\to \infty}P_{B_n}\left(f\right)=f. 
\end{align*}
By $AP\left(\XB\right)$ we denote the set of elements of $f$ that are approximable by proyections. If $AP\left(\XB\right)=\XX$, we say that $\XB$ has the approximability by proyections property (or APP for short). 
\end{definition}
\begin{remark}\label{remarkAPP}\rm 
Note that the APP is a mild condition on a basis. For instance, in addition to Schauder bases and quasi-greedy Markushevich bases, any Markushevich basis that is naturally associated with a finite-dimensional decomposition has this property, and so does a Markushevich basis that is $\n$-quasi-greedy through a sequence $\n=\left(n_k\right)_{k}\subset \NN$ (see \cite{O2018}).
\end{remark}

In our next result, we construct embeddings from weighted Lorentz sequence spaces to the greedy classes - allowing that the latter may not be linear spaces -, for bases that are truncation quasi-greedy and have the APP. This relaxes the unconditionality hypothesis from \cite[Theorem 3.1]{GHN2012}. To prove it, we will use the following result from the literature.

 \begin{lemma}[\cite{AABW2021}*{Theorem 2.2}]\label{lemmaAp}
Let $\XX$ be a $p$-Banach space. Given $A\in \NN^{<\infty}$ and a family $\left\lbrace f_n: n\in A\right\rbrace\subset\XX$, we have
\begin{align*}
&\left\Vert \sum_{n\in A}a_nf_n\right\Vert_{\XX}\le \mathbf{A_p}\sup_{B\subset A} \left\Vert \sum_{n\in B}f_n\right\Vert_{\XX} &&\forall \left(a_n\right)_{n\in A}\subset [0,1]. 
\end{align*}
where
\begin{align*}
\mathbf{A_p}:=\frac{1}{\left(2^p-1\right)^{\frac{1}{p}}}.
\end{align*}
\end{lemma}
 
Now we can prove our result; part of the argument is from the proof of \cite[Theorem 3.1]{GHN2012}.

\begin{lemma}\label{lemmajackson1}
Let $\XB$ be a basis of a $p$-Banach space $\XX$. Fix $w \in\WW_{i,d,+}$. The following hold: 
\begin{enumerate}[\rm (i)]
\item \label{forinf}
There is a positive constant $C=C(w,p. \XB)>0$ such that for all $m\in \NN_0$ and all $f\in AP\left(\XB\right)$, 
\begin{align*}
&\gamma_m\left(f\right)\leq \frac{C}{w\left(m\right)}\Vert f\Vert_{\ell^{\infty}_{w\hr}}, 
\end{align*}
where $w\left(0\right):=w\left(1\right)$. 
\item \label{therest} For each $0<q\le \infty$, there is $C=C\left(p,q,w,\XB\right)$ such that 
\begin{align*}
&\left\Vert f\right\Vert_{\GG^{w}_q}\le C \left\Vert f\right\Vert_{\ell^q_{w \hr}}&& \forall f\in AP\left(\XB\right). 
\end{align*} 
Hence, if $\XB$ has the APP, then $\ell^q_{w \hr}\hookrightarrow \GG^{w}_q$. 
\end{enumerate}

\end{lemma}
\begin{proof}
Choose $C_1\ge 1$ so that $w$ is $C_1$-doubling. \\
Take an element $f=\sum_{j=1}^\infty \xx_j^*\left(f\right)\xx_j\in \langle \XB\rangle$ and let $\rho$ be a greedy ordering for $f$. Fix a number $m\in \NN$ and denote $\lambda_j=2^j m $ and $f_j^*=\vert\xx_{\rho(j)}^*\left(f\right)\vert$ for each $j$, and $A:=\left\lbrace \rho\left(j\right): 1\le j\le m\right\rbrace$. Considering that for each $1\le n\le \lambda_j$, we have $\hr(n)\le 2^{\frac{1}{p}-1} \hr(\lambda_{j})$, it follows that
	\begin{align}
	\left \Vert f-P_{A}\left(f\right)\right \Vert_{\XX}^p=&\left\Vert \sum_{k=m+1}^{\infty}\xx_{\rho(k)}^*\left(f\right)\xx_{\rho(k)} \right\Vert_{\XX}^p\le  \sum_{j=0}^{\infty}\left \Vert\sum_{\lambda_j< k\le \lambda_{j+1}}\xx_{\rho(k)}^*\left(f\right)\xx_{\rho(k)}\right\Vert^p_{\XX} \nonumber\\
	 \le &2^{1-p}(\mathbf{A_p})^p\sum_{j=0}^\infty (w(\lambda_j)f_{\lambda_j}^*\hr(\lambda_j))^p\frac{1}{(w(\lambda_j))^p}\nonumber\\
		\le& 2^{1-p}(\mathbf{A_p})^p\Vert f\Vert_{\ell_{w\hr}^\infty}^p\sum_{j=0}^\infty \frac{1}{(w(2^j\left(m\right)))^p}.\label{k100}
	\end{align}
(the computation above is justified even if $\XB$ is not quasi-greedy because $f\in \langle \XB\rangle$). \\		
Now, for any $j\in \NN$, since $i_w>0$, by Theorem~\ref{theoremLRPURP}, we can take any $\alpha\in (0,i_w)$ to obtain that

	\begin{eqnarray*}
		\frac{w\left(2^jm\right)}{w\left(m\right)}\geq C_\alpha 2^{j\alpha}.
	\end{eqnarray*}

	\noindent Hence, 
	\begin{eqnarray}\label{k200}
		\frac{1}{w\left(2^jm\right)}\leq C_\alpha^{-1}2^{-j\alpha}\frac{1}{w\left(m\right)}.
	\end{eqnarray}
		
	\noindent Applying \eqref{k200} in \eqref{k100} we get
	\begin{eqnarray*}
		\Vert f-P_A\left(f\right)\Vert^p_{\XX} \leq (C(p,w,C_{\alpha}))^p\Vert f\Vert_{\ell_{w\hr}^\infty}^p\frac{1}{(w\left(m\right))^p}.
	\end{eqnarray*}
Taking supremum over all greedy orderings for $f$, we obtain the upper bound for $\gamma_	m\left(f\right)$. We have yet to consider the case $m=0$, that is we have to prove that 
\begin{align*}
\left\Vert f\right\Vert_{\XX}\le& C(p,w,\XB) \left\Vert f\right\Vert_{\ell^{\infty}_{w \hr}}. 
\end{align*}
But this follows easily from the case $m=1$: 
\begin{align*}
\left\Vert f\right\Vert_{\XX}\le& \gamma_1\left(f\right)+\left\Vert \xx_{\rho\left(1\right)}^*\left(f\right)\xx_{\rho\left(1\right)}\right\Vert_{\XX}\le \frac{C\left(p,w,C_{\alpha}\right)}{w\left(1\right)} \left\Vert f\right\Vert_{\ell^{\infty}_{w \hr}}+\alpha_1 f_1^*\\
=&\frac{C\left(p,w,C_{\alpha}\right)}{ w\left(1\right)}  \left\Vert f\right\Vert_{\ell^{\infty}_{w \hr}}+\frac{w\left(1\right)\hr\left(1\right)f_1^*}{w\left(1\right)}\\
\le& \left(\frac{C\left(p,w,C_{\alpha}\right)+1}{w\left(1\right)} \right)\left\Vert f\right\Vert_{\ell^{\infty}_{w \hr}}.
\end{align*}
This completes the proof of \ref{forinf} for $f\in \langle \XB\rangle$. Now suppose $f\in AP\left(\XB\right)\setminus \langle \XB\rangle$, and choose a sequence $\left(B_n\right)_{n\in \NN}\subset \NN^{<\infty}$ so that
\begin{align*}
&\lim_{n\to \infty}P_{B_n}\left(f\right)=f. 
\end{align*}
Given $m\in \NN_0$, choose $A_m\in GS\left(f,m\right)$ so that
\begin{align*}
\gamma_m\left(f\right)=&\left\Vert f-P_{A_m}\left(f\right)\right\Vert_{\XX}. 
\end{align*}
Choose $n_m\in \NN$ so that for each $n\ge n_m$, we have $A_m\in GS\left(P_{B_n}\left(f\right),m\right)$ (which is possible because each element in the support of $f$ is eventually in all of the $B_n$'s). Since $P_{B_n}\left(f\right)\in \langle \XB\rangle$ for each $n\in \NN$, we get
\begin{align*}
\gamma_m\left(f\right)=&\left\Vert f-P_{A_m}\left(f\right)\right\Vert_{\XX}=\lim_{\substack{n\to \infty\\ n\ge n_m}}\left\Vert  P_{B_n}\left(f\right)-P_{A_m}\left( P_{B_n}\left(f\right)\right)\right\Vert_{\XX}\\
\le& \limsup_{\substack{n\to \infty\\ n\ge n_m}}\gamma_m\left(P_{B_n}\left(f\right)\right)\le \frac{C}{w\left(m\right)}\limsup_{\substack{n\to \infty\\ n\ge n_m}}\left\Vert P_{B_n}\left(f\right)\right\Vert_{\ell^{\infty}_{w \hr}}\\
\le& \frac{C}{w\left(m\right)}\left\Vert f\right\Vert_{\ell^{\infty}_{w \hr}},
\end{align*}
where $C=C\left(p,w,\XB\right)$ is the constant obtained previously. \\
Now we prove \ref{therest}: If $q=\infty$, this follows at once from \ref{forinf}. If $0<q<\infty$,  let $\mu=\min\lbrace q,p\rbrace$. Given $f \in \ell^q_{w \eta}$ with finite support and $\rho$ a greedy ordering for $f$, for each $n\in \NN$ let $A_n:=\left\lbrace \rho\left(k\right): 1\le k\le n  \right\rbrace$ and $f_n^*=\vert\xx_{\rho(n)}^*\left(f\right)\vert$. Now write $f=\sum_{j=-1}^\infty\sum_{2^j<k\leq 2^{j+1}}\xx_{\rho(k)}^*\left(f\right)\xx_{\rho(k)}$. Given $m\in \NN_0$ and $2^m< n\le 2^{m+1}$, define $\left(a_k\right)_{k\in \NN}$ by 
\begin{align*}
a_k:=&
\begin{cases}
0 & \text{if } 1\le k\le n;\\
1 & \text{if } k>n. 
\end{cases}
\end{align*}
We have 
\begin{align*}
\left \Vert f-P_{A_n}\left(f\right)\right\Vert_{\XX}^\mu=&\left\Vert \sum_{j=m}^{\infty}\sum_{k=2^j+1}^{2^{j+1}}a_k\xx_{\rho\left(k\right)}^*\left(f\right)\xx_{\rho\left(k\right)}\right\Vert_{\XX}^{\mu}\le \sum_{j=m}^{\infty}\left\Vert \sum_{k=2^j+1}^{2^{j+1}}a_k\xx_{\rho\left(k\right)}^*\left(f\right)\xx_{\rho\left(k\right)}\right\Vert_{\XX}^{\mu}\\
\le& \mathbf{A_p}^\mu \sum_{j=m}^{\infty}\left\vert \xx_{\rho\left(2^{j}\right)}\right\vert^{\mu}
\sup_{\substack{B\subset \left\lbrace \rho\left(k\right): 2^{j}+1\le k\le 2^{j+1}\right\rbrace\\ \varepsilon\in \EE^{\NN}}} \left\Vert\Ind_{\varepsilon, B}\right\Vert_{\XX}^{\mu}\\
\le &2^{\mu\left(\frac{1}{p}-1\right)} \mathbf{A_p}^\mu \sum_{j=m}^{\infty}\left(f_{2^j}^*\hr\left(2^{j}\right)\right)^{\mu}, 
\end{align*}

so taking supremum we get
\begin{align*}
&\gamma_n\left(f\right)\le 2^{\mu\left(\frac{1}{p}-1\right)}\mathbf{A_p}\left(\sum_{j=m}^{\infty}\left(f_{2^j}^*\hr\left(2^{j}\right)\right)^{\mu}\right)^{\frac{1}{\mu}}&&\forall m\in \NN_0,\forall 2^{m}<n\le 2^{m+1}.
\end{align*}
Since $q/\mu\ge 1$, we can apply Minkowski’s inequality to obtain
\begin{align}
&\left(\sum_{n=2}^{\infty}\left(w\left(n\right)\gamma_n\left(f\right)\right)^{q}\frac{1}{n}\right)^{\frac{1}{q}}\le \left(\sum_{m=0}^{\infty}\frac{1}{2^m}\sum_{2^{m}< n\le 2^{m+1}}\left(w\left(n\right)\gamma_n\left(f\right)\right)^{q}\right)^{\frac{1}{q}}\nonumber\\
&\le 2^{\mu\left(\frac{1}{p}-1\right)}C_1\mathbf{A_p}\left( \left(\sum_{m=0}^{\infty}\left(\sum_{j=0}^{\infty}\left(w\left(2^{m}\right)f_{2^{m+j}}^*\hr\left(2^{m+j}\right)\right)^{\mu}\right)^{\frac{q}{\mu}}\right)^{\frac{\mu}{q}}\right)^{\frac{1}{\mu}}\nonumber\\
&\le 2^{\mu\left(\frac{1}{p}-1\right)}C_1\mathbf{A_p}\left(\sum_{j=0}^{\infty}\left(\sum_{m=0}^{\infty}\left(w\left(2^{m}\right)f_{2^{m+j}}^*\hr\left(2^{m+j}\right)\right)^{q}\right)^{\frac{\mu}{q}}\right)^{\frac{1}{\mu}}. \label{k1200}
\end{align}	
Note that if $l=m+j$, taking $\alpha$ as in the proof of \ref{forinf} we get 
\begin{align*}
\frac{w\left(2M\right)}{w\left(2^l\right)}\le C_{\alpha}^{-1} 2^{-j\alpha}. 
\end{align*}
	From this, \eqref{k1200} and Remark~\ref{remarkchequear}, we get
	\begin{eqnarray*}
		\left(\sum_{n=2}^{\infty}\left(w\left(n\right)\gamma_n\left(f\right)\right)^{q}\frac{1}{n}\right)^{\frac{1}{q}}&\leq& C(p,q,w)\left[\sum_{j=0}^\infty 2^{-j\alpha \mu}\left( \sum_{l=j}^\infty  [f_{2^{l}}^* w(2^l)\hr(2^{l})]^q\right)^{\frac{\mu}{q}}\right]^{\frac{1}{\mu}}\\
	&\leq& C'\Vert f\Vert_{\ell^q_{w \hr}},
	\end{eqnarray*}
	where $C'=C'(p,q,w,\XB)$. Hence, 
	\begin{align*}
\left\Vert f\right\Vert_{\GG^w_q}=&	\left\Vert f\right\Vert_{\XX}+\left(\sum_{n=1}^{\infty}\left(w\left(n\right)\gamma_n\left(f\right)\right)^{q}\frac{1}{n}\right)^{\frac{1}{q}}\\
\le& \left\Vert f\right\Vert_{\XX}+ 2^{\frac{1}{q}}\left(\left(\sum_{n=2}^{\infty}\left(w\left(n\right)\gamma_n\left(f\right)\right)^{q}\frac{1}{n}\right)^{\frac{1}{q}}+w\left(1\right)\gamma\left(1\right) \right)\\
	\le& 2^{\frac{1}{q}}\left(C'\Vert f\Vert_{\ell^q_{w \hr}}+\left(1+ w\left(1\right)\left(\alpha_3+1\right)\right)\left\Vert f\right\Vert_{\XX}\right)\\
=&C'' \left\Vert f\right\Vert_{\XX}+C''\left\Vert f\right\Vert_{\ell^q_{w \hr}},
	\end{align*}
where $C''=C''\left(p,q,w,\XB \right)$. To finish the proof for $f\in \langle \XB\rangle$, we need to prove that $\left\Vert f\right\Vert_{\XX}\lesssim \left\Vert f\right\Vert_{\ell^{q}_{w \hr}}$. However, since $\ell^{q}_{w \hr}\hookrightarrow \ell^{\infty}_{w \hr}$ for all $0<q<\infty$, this follows at once from the case $q=\infty$. Now fix $f\in AP\left(\XB\right)\setminus \langle \XB\rangle$, and choose a sequence $\left(B_n\right)_{n\in \NN}\subset \NN^{<\infty}$ so that
\begin{align*}
&\lim_{n\to \infty}P_{B_n}\left(f\right)=f.
\end{align*}
As before, for each $j\in \supp\left(f\right)$, there is $n\left(j\right)$ so that $j\in B_n$ for all $n\ge n\left(j\right)$. Hence, given $m\in \NN$, there is $n_m\in \NN$ such  that 
\begin{align*}
&GS\left(f,k\right)=GS\left(P_{B_n}\left(f\right), k\right)&&\forall 1\le k\le m,\forall n\ge n_m. 
\end{align*}
It follows that for each $m\in \NN$, 
\begin{align*}
&\left\Vert f\right\Vert_{\XX}+\left(\sum_{k=1}^{m}\left(w\left(k\right)\gamma_k\left(f\right)\right)^{q}\frac{1}{k}\right)^{\frac{1}{q}}\\
&=\lim_{\substack{n\to\infty\\n\ge n_m}}\left\Vert P_{B_n}\left(f\right)\right\Vert_{\XX}+\left(\sum_{k=1}^{m}\left(w\left(k\right)\gamma_k\left( P_{B_n}\left(f\right)\right)\right)^{q}\frac{1}{k}\right)^{\frac{1}{q}}\\
&\le\limsup_{\substack{n\to\infty\\n\ge n_m}}\left\Vert  P_{B_n}\left(f\right)\right\Vert_{\GG^{w}_q}\le C\left(p,q,w,\XB\right)\limsup_{\substack{n\to\infty\\n\ge n_m}}\left\Vert  P_{B_n}\left(f\right)\right\Vert_{\ell_{w\hr}^q}\\
&\le C\left(p,q,w,\XB\right)\left\Vert  f\right\Vert_{\ell_{w\hr}^q}. 
\end{align*}
As $m$ is arbitrary,  the proof is complete. 
\end{proof}
In \cite{GN2001}, the authors give a full characterization of the classical approximation spaces of greedy bases \cite[Corollary 1]{GN2001} and of the greedy classes of almost greedy bases \cite[Theorem 5.1]{GN2001}. Our results in this section allow us to characterize both the approximation spaces and the greedy classes of a wider type of bases that includes not only all almost greedy bases but also some bases that fail to be quasi-greedy (see for example \cite[Proposition 3.17]{AABBL2023}). 
\begin{lemma}\label{lemmasqueeze1}Let $\XB$ be a basis of a $p$-Banach space $\XX$. If $\XB$ is truncation quasi-greedy and has the APP, the following are equivalent: 
\begin{enumerate}[\rm (i)]
\item \label{equivalences} For each $w\in \WW_{i,d,+}$  and every $0<q\le \infty$, 
\begin{align}
\GG^{w}_q\approx \A^{w}_q\approx \ell^q_{w\hr}. \label{characterization1}
\end{align}
\item \label{equivalences1} There is $0<q\le \infty$ and $w\in \WW_{i,d,+,-}$ such that 
\begin{align*}
\GG^{w}_q\approx \A^{w}_q\approx \ell^q_{w\hr}. 
\end{align*}
\item \label{superdemocracy1}$\XB$ is superdemocratic. 
\item \label{democracy1}$\XB$ is democratic.
\end{enumerate} 
\end{lemma}
\begin{proof}
As $\XB$ is truncation quasi-greedy, 
\ref{superdemocracy1} is equivalent to \ref{democracy1} (see for example \cite[Proposition 4.16]{AABW2021}). \\
To prove that \ref{equivalences1} entails \ref{superdemocracy1},  apply  Lemma~\ref{lemmabernstein2} with $\eta=\hr$. \\
Finally, if \ref{superdemocracy1} holds, we can take $\eta=\hr$ in 
Lemma~\ref{lemmaBernstein1}. From this, Lemma~\ref{lemmajackson1} and the fact that $\GG^{w}_q\hookrightarrow \A^{w}_q$ we obtain \ref{equivalences}.
\end{proof}
Lemma~\ref{lemmasqueeze1} extends the characterizations of greedy classes and approximation spaces from \cite{GN2001}, but does not address the question of equality $\A^w_q=\GG^{w}_q$. In the next section we will remove this limitation. 

\section{Classical, greedy and Chebyshev-greedy approximation classes.}\label{sectiongreedy=classical}
The main purpose of this section is the study of conditions under which the approximation classes are the same. Note that  $\GG^{w}_q\hookrightarrow \CG^w_1\hookrightarrow \A^{w}_q$ by \eqref{keyrel1}. For weights with $i_w>0$, it is known that if a basis $\XB$ is unconditional and $\GG^{w}_q\approx \A^{w}_q$, then $\XB$ is also democratic, and thus greedy \cite[Theorem 2.14]{BCH2023} (see also \cite{GHN2012}, \cite{W2014}). We will prove democracy under the weaker hypotheses that that $\XB$ is unconditional for constant coefficients and $\GG^w_q=\A^w_q$, and obtain similar results for the Chebyshev classes. Our main result is a characterization of the truncation quasi-greedy bases with the APP for which $\GG^w_q=\A^w_q$ or $\CG^w_q=\A^w_q$. Before we can prove these results, we need some  technical lemmas. The first one will allow us to estimate $\left\Vert f\right\Vert_{\GG^{w}_{q}}$ in terms of $\left\Vert f\right\Vert_{\A^w_q}$, for $f$ in a subset of $\GG^{w}_q$ that is sufficiently large for our purposes - an estimate that is the key to removing the restriction involving equivalent quasi-norms and replacing the condition $\GG^{w}_q\approx \A^{w}_q$ with $\GG^{w}_q=\A^w_q$ in our results. 

\begin{lemma}\label{lemmasortofclosed}Let $\XX$ be a $p$-Banach space $\XX$ with a basis $\XB$, $0< q\le \infty$, and $w$ a weight. For each $t>0$, define
\begin{align*}
U_t=U_t\left[\GG^{w}_q\right]:=\left\lbrace f\in  \GG^{w}_{q}: \left\Vert f\right\Vert_{\GG^{w}_{q}}\le t \right\rbrace. 
\end{align*}
Then 
\begin{enumerate}[\rm (i)]
\item \label{firstpartforbaire} For each $t>0$ we have
\begin{align}
&\left\Vert f\right\Vert_{\XX}+\left(\sum_{n\in NSG\left(f\right)}\left(w\left(n\right)\gamma_n\left(f\right) \right)^{q}\frac{1}{n}\right)^{\frac{1}{q}}\le t&&\forall f\in \overline{U_t}^{\left\Vert \cdot\right\Vert_{\XX}} \label{strictgreedy}
\end{align}
if $0<q<\infty$, and 
\begin{align}
\left\Vert f\right\Vert_{\XX}+\sup_{n\in NSG\left(f\right)}w\left(n\right)\gamma_n\left(f\right) \le t&&\forall f\in \overline{U_t}^{\left\Vert \cdot\right\Vert_{\XX}}\label{strictgreedyb}
\end{align}
if $q=0$ or $q=\infty$.\\
 Hence, if $f\in  SG\left(\XB\right)\cap \overline{U_t}^{\left\Vert \cdot\right\Vert_{\XX}}$, then $f\in U_t$. 
\item \label{secondpartbaire}Let  $\YY=\GG^{w}_q$.   Suppose that 
\begin{itemize}
\item $\YY$ is a linear space.
\item There is a quasi-norm $\left\Vert \cdot\right\Vert_{\circ}$ defined on $\YY$ such that $\left\Vert f\right\Vert_{\XX}\lesssim \left\Vert f\right\Vert_{\circ}$ and $\left(\YY, \left\Vert \cdot\right\Vert_{\circ}\right)$ is complete.
\item $\XB$ is bounded with respect to $\left\Vert \cdot\right\Vert_{\circ}$. 
\end{itemize} 
Then, there is $t>0$ such that 
\begin{align}
&\left\Vert f\right\Vert_{\GG^{w}_{q}}\le t\left\Vert f\right\Vert_{\left\Vert \cdot\right\Vert_{\circ}}&&\forall f\in \langle \XB\rangle \cap SG\left(\XB\right). \label{keybound2}
\end{align}
\end{enumerate}
\end{lemma}
\begin{proof}
\ref{firstpartforbaire}  
To prove \eqref{strictgreedy} and \eqref{strictgreedyb}, we may assume that $NSG\left(f\right)\not=\emptyset$.  Choose $\left(f_k\right)_{k\in \NN}\subset U_t$ so that 
\begin{align*}
f_k\xrightarrow[k\to \infty]{\left\Vert \cdot\right\Vert_{\XX}}f,
\end{align*}
and fix $m\in \NN$ so that $NSG\left(f\right)_{\le m}\not=\emptyset$. Given $n\in NSG\left(f\right)_{\le m}$, let $A$ be the only greedy set of $f$ of cardinality $n$, and choose $\delta>0$ so that 
\begin{align*}
&\left\vert \xx_j^*\left(f\right)\right\vert\ge \delta+\left\vert \xx_l^*\left(f\right)\right\vert &&\forall j\in A, l\not\in A. 
\end{align*}
Since $\XB^*$ is bounded, there is $k_0\in \NN$ such that 
\begin{align*}
&\left\vert \xx_j^*\left(f_k\right)\right\vert\ge \frac{\delta}{2}+\left\vert \xx_l^*\left(f_k\right)\right\vert &&\forall j\in A, l\not\in A,\forall k\ge k_0. 
\end{align*}
Hence, for each $k\ge k_0$, $A\in SGS\left(f_k,n\right)$, so $n\in NSG\left(f_k\right)$ and
\begin{align*}
&\gamma_n\left(f_k\right)=\left\Vert f_k-P_A\left(f_k\right)\right\Vert_{\XX}
\end{align*}
Since $\gamma_n\left(f\right)=\left\Vert f-P_A\left(f\right)\right\Vert_{\XX}$, this implies that 
\begin{align*}
\gamma_n\left(f_k\right)\xrightarrow[k\to \infty]{}\gamma_n\left(f\right). 
\end{align*}
Considering all $n\in NSG\left(f\right)_{\le m}$, if $0<q<\infty$ it follows that  
\begin{align*}
&\left\Vert f\right\Vert_{\XX}+\left(\sum_{n\in NSG\left(f\right)_{\le m}}\left(w\left(n\right)\gamma_n\left(f\right) \right)^{q}\frac{1}{n}\right)^{\frac{1}{q}}\\
&=\lim_{k\to \infty}\left\Vert f_k\right\Vert_{\XX}+\left(\sum_{n\in NSG\left(f_k\right)_{\le m}}\left(w\left(n\right)\gamma_n\left(f_k\right) \right)^{q}\frac{1}{n}\right)^{\frac{1}{q}}\le t. 
\end{align*}
Since $m$ is arbitrary, we have proved \eqref{strictgreedy} for such $q$. The proof of \eqref{strictgreedyb} requires only minor adjustments. \\
Now fix $f_0\in SG\left(\XB\right)\cap \overline{U_t}^{\left\Vert \cdot\right\Vert_{\XX}}$, and suppose for the sake of contradiction that $f\not\in U_t$.  Then, if $0<q<\infty$, there are $\epsilon>0$ and $m_0\in \NN$ such that 
\begin{align}
&\left\Vert f_0\right\Vert_{\XX}+\left(\sum_{n=1}^{m_0}\left(w\left(n\right)\gamma_n\left(f_0\right) \right)^{q}\frac{1}{n}\right)^{\frac{1}{q}}=t+\epsilon. \label{forcont1new}
\end{align} 
Let $B_0:=\supp\left(f_0\right)$ and $m_1:=\left\vert B_0\right\vert$. Since $f_0\in SG\left(\XB\right)$, by Remark~\ref{remarkonlyonegreedy} and \eqref{strictgreedy} we have 
\begin{align}
&\left\Vert f_0\right\Vert_{\XX}+\left(\sum_{n=1}^{m_1}\left(w\left(n\right)\gamma_n\left(f_0\right) \right)^{q}\frac{1}{n}\right)^{\frac{1}{q}}\le t.\nonumber
\end{align}
It follows that $m_0>m_1$.\\
Choose $\left(f_k\right)_{k\in \NN}$ as in the proof of \eqref{strictgreedy}, for $f=f_0$. \\
In the case $B_0\not=\emptyset$, the same argument of the proof of \eqref{strictgreedy} gives that there is $k_0\in \NN$ such that, for each $k\ge k_0$ and each $1\le n\le m_1$, there is a unique $A_{n}\in GS\left(f_k,n\right)=GS\left(f_0,n\right)$, and 
\begin{align}
&\gamma_n\left(f_k\right)=\left\Vert f_k-P_{A_{n}}\left(f_k\right)\right\Vert_{\XX}\xrightarrow[k\to \infty]{} \left\Vert f_0-P_{A_{n}}\left(f_0\right)\right\Vert_{\XX}=\gamma_n\left(f_0\right).\label{forcont4new}
\end{align}
For each $m_1\le n\le m_0$ and each $k\ge k_0$, choose $A_{k,n}\in GS\left(f_k, n\right)$ so that $\gamma_n\left(f_k\right)=\left\Vert f_k-P_{A_{k,n}}\left(f_k\right)\right\Vert$. Since $B_0=A_{m_1} \in SGS\left(f_k,m_1 \right)$, it follows that for such $k$ and $n$, $B_0\subset A_{n,k}$, so 
\begin{align*}
\left\vert \gamma_n^p\left(f_k\right)-\gamma_n^p\left(f_0\right)\right\vert=&\left\vert \left\Vert f_k-P_{B_0}\left(f_k\right) -P_{A_{k,n}\setminus B_{0}}\left(f_k\right)\right\Vert^p-\left\Vert f_0-P_{B_0}\left(f_0\right)\right\Vert^p\right\vert\\
\le& \left\vert \left\Vert f_k-P_{B_0}\left(f_k\right)\right\Vert^p-\left\Vert f_0-P_{B_0}\left(f_0\right)\right\Vert^p\right\vert+ \left\Vert P_{A_{k,n}\setminus B_{0}}\left(f_k-f\right)\right\Vert^p \\
\le&  \left\vert \left\Vert f_k-P_{B_0}\left(f_k\right)\right\Vert^p-\left\Vert f_0-P_{B_0}\left(f_0\right)\right\Vert^p\right\vert+ m_0\alpha_3\left[\XB, \XX, \left\Vert \cdot\right\Vert_{\XX}\right]^p\left\Vert f_0-f_k\right\Vert^p. 
\end{align*}
Taking limit for $k\to \infty$, it follows that 
\begin{align}
&\gamma_n\left(f_k\right)\rightarrow \gamma_n\left(f_0\right) &&\forall m_1 \le n\le m_0. \label{forcont5new}
\end{align}
Combining \eqref{forcont1new}, \eqref{forcont4new} and \eqref{forcont5new}, it follows that for sufficiently large $k$, 
\begin{align*}
&\left\Vert f_k\right\Vert_{\XX}+\left(\sum_{n=1}^{m_0}\left(w\left(n\right)\gamma_n\left(f_k\right) \right)^{q}\frac{1}{n}\right)^{\frac{1}{q}}>t, 
\end{align*} 
a contradiction. \\
The case $B_0=\emptyset$ is simpler: we have $\gamma_n\left(f_0\right)=\left\Vert f_0\right\Vert$ for all $n\in \NN$. On the other hand, 
\begin{align*}
\sup_{\left\vert B\right\vert\le m_0}\left\Vert P_B\left(f_k\right)\right\Vert^p \le& m_0\alpha_3\left[\XB, \XX, \left\Vert \cdot\right\Vert_{\XX}\right]^p\left\Vert f_0-f_k\right\Vert^p\xrightarrow[k\to \infty]{}0, 
\end{align*} 
it follows that 
\begin{align*}
&\gamma_{n}\left(f_k\right)\xrightarrow[k\to \infty]{} \left\Vert f_0\right\Vert =\gamma_{n}\left(f_0\right) &&\forall 1\le n\le m_0, 
\end{align*}
which, as before, entails a contradiction.\\
Finally the case $q=\infty$ is proved by a simpler version of the above argument. \\
\ref{secondpartbaire}
We may assume that $\left\Vert f\right\Vert_{\XX}\le \left\Vert f\right\Vert_{\circ}$ for all $f\in \YY$. Since 
\begin{align*}
\overline{\langle \XB\rangle}^{\left\Vert \cdot\right\Vert_{\circ}}=&\bigcup_{n\in \NN}\overline{\langle \XB\rangle}^{\left\Vert \cdot\right\Vert_{\circ}}\cap \overline{U_n}^{\left\Vert \cdot\right\Vert_{\circ}},
\end{align*}
by Baire category theorem there are $n_0\in \NN$, $f_0\in \overline{\langle \XB\rangle}^{\left\Vert \cdot\right\Vert_{\circ}}$ and $0<r_0\le 1$  such that 
\begin{align}
\overline{B}_{\left\Vert \cdot\right\Vert_{\circ}}\left(f_0,r_0\right)\cap  \overline{\langle \XB\rangle}^{\left\Vert \cdot\right\Vert_{\circ}}\subset  \overline{\langle \XB\rangle}^{\left\Vert \cdot\right\Vert_{\circ}}\cap \overline{U_{n_0}}^{\left\Vert \cdot\right\Vert_{\circ}}.\nonumber
\end{align}
By density, we may assume that $f_0\in \langle \XB\rangle$. Moreover, by a small perturbations argument, we may further assume that $f_0\in SG\left(\XB\right)$, hence $f_0\in U_{n_0}$ by \ref{firstpartforbaire}. Additionally, \ref{firstpartforbaire} entails that 
\begin{align}
\overline{B}_{\left\Vert \cdot\right\Vert_{\circ}}\left(f_0,r_0\right)\cap \langle \XB\rangle \cap SG\left(\XB\right)\subset U_{n_0}. \label{onebaire3}
\end{align}
The next step in the proof consists in finding $0<r\le r_0$ and $M\ge n_0$ so that when we substitute $r$ for $r_0$ and $M$ for $n_0$ in \eqref{onebaire3}, we may also substitute $0$ for $f_0$. To prove this, assume $f_0\not=0$, and let $A:=\supp\left(f_0\right)$ and $m_0:=\left\vert A\right\vert$. By an inductive argument we only need to find $0<r_1\le r_0$ and $n_1\ge n_0$ so that when we replace  $r_0$ and $n_0$ by $r_1$ and $n_1$ respectively, we may also replace $f_0$ with some $f_1\in SG\left(\XB\right)\cap U_{n_1}\cap \langle \XB\rangle$ whose support has cardinality no greater $m_0-1$. Since the ordering of the basis plays no role in our arguments, we may assume that $SGS\left(f_0,n\right)=\left\lbrace \left\lbrace 1,\dots,n\right\rbrace\right\rbrace$ for each $1\le n\le m_0$. Let
\begin{align*}
n_1:=&\floor{ 2^{\frac{1}{p}}\left(1+2\alpha_1^p\alpha_2^p\right)^{\frac{1}{p}} n_0+2^{\frac{1}{p}}\alpha_3\left\Vert f_0\right\Vert_{\XX}}+2,\\
f_1:=&f_0-\xx_1^*\left(f_0\right)\xx_1,\\
A_1:=&A\setminus \left\lbrace 1\right\rbrace,\\
r_1:=&\frac{r_0}{2\left(1+\alpha_2\right)}\min\left\lbrace 1,\min_{\substack{n\in A\\j\in \NN\setminus\left\lbrace n\right\rbrace}}\left\lbrace \left\vert \left\vert\xx_n^*\left(f_0\right)\right\vert-\left\vert\xx_j^*\left(f_0\right)\right\vert\right\vert \right\rbrace\right\rbrace.
\end{align*}
We claim that
\begin{align}
\overline{B}_{\left\Vert \cdot\right\Vert_{\circ}}\left(f_1,r_1\right)\cap SG\left(\XB\right)\cap \langle \XB\rangle\subset U_{n_1} \label{keyclaim3}
\end{align}
(which completes the inductive step). To see this, fix $$0\not=f\in \overline{B}_{\left\Vert \cdot\right\Vert_{\circ}}\left(f_1,r_1\right)\cap SG\left(\XB\right)\cap \langle \XB\rangle.$$ If $2\le n\le  m_0$ and $k>n$ or $k=1$, 
\begin{align*}
\left\vert \xx_n^*\left(f\right)\right\vert\ge& \left\vert \xx_n^*\left(f_1\right)\right\vert-\left\vert \xx_n^*\left(f-f_1\right)\right\vert\ge \left\vert \xx_n^*\left(f_1\right)\right\vert-\alpha_2r_1\\
\ge&  2\left(1+\alpha_2\right)r_1+ \left\vert \xx_k^*\left(f_1\right)\right\vert-\alpha_2r_1\\
>&\left(1+\alpha_2\right)r_1+\left\vert \xx_k^*\left(f\right)\right\vert-\left\vert \xx_k^*\left(f_1-f\right)\right\vert>\left\vert \xx_k^*\left(f\right)\right\vert.
\end{align*}
Hence, for each $2\le n\le m_0$,
$$GS\left(f,n-1\right)=SGS\left(f,n-1\right)=\left\lbrace \left\lbrace 2,\dots, n\right\rbrace\right\rbrace.$$ Similarly, if $g=f+\xx_1^*\left(f_0\right)\xx_1$, for each $1\le n\le  m_0$ we have $GS\left(g,n\right)=SGS\left(g,n\right)=\left\lbrace \left\lbrace 1,\dots, n\right\rbrace\right\rbrace$. Thus, if $j\in \NN$ and $D\in GS\left(f,j\right)$, there are two possible cases: 
\begin{itemize}
\item If $1\in D$, then $D\in  GS\left(g,j\right)$. Thus, 
\begin{align*}
\gamma_j\left(f\right)=&\left\Vert f-P_{D}\left(f\right)\right\Vert_{\XX}=\left\Vert g-P_{D}\left(g\right)\right\Vert_{\XX}=\gamma_j\left(g\right).
\end{align*}
\item If $1\not\in D$, then there is $k\in D$ such that $E:=\left\lbrace 1\right\rbrace\cup D\setminus\left\lbrace k\right\rbrace\in GS\left(g,j\right)$. Then 
\begin{align*}
\gamma_j\left(g\right)=&\left\Vert g-P_{E}\left(g\right)\right\Vert_{\XX} \ge \frac{1}{\alpha_2}\left\vert \xx_k^*\left(g\right)\right\vert=\frac{1}{\alpha_2}\left\vert \xx_k^*\left(f\right)\right\vert\ge \frac{1}{\alpha_2}\left\vert \xx_1^*\left(f\right)\right\vert.
\end{align*}
Hence, 
\begin{align*}
\gamma_j\left(f\right)=&\left\Vert f-P_{D}\left(f\right)\right\Vert_{\XX}\\
\le& \left(\left\Vert g-P_{E}\left(g\right)\right\Vert_{\XX}^p+\left\Vert \xx_k^*\left(f\right)\xx_k\right\Vert_{\XX}^p+\left\Vert \xx_1^*\left(f\right)\xx_1\right\Vert_{\XX}^p\right)^{\frac{1}{p}}\\
\le& \left(1+2\alpha_1^p\alpha_2^p\right)^{\frac{1}{p}}\gamma_j\left(g\right).
\end{align*}
\end{itemize}
Since $g\in \overline{B}_{\left\Vert \cdot\right\Vert_{\circ}}\left(f_0,r_0\right)\cap SG\left(\XB\right)\cap\langle \XB\rangle\subset U_{n_0}$, in the case $0<q<\infty$ we have
\begin{align*}
\left\Vert f\right\Vert_{\GG^{w}_{q}}=&\left\Vert f\right\Vert_{\XX}+\left(\sum_{n\in \NN}\left(w\left(n\right)\gamma_n\left(f\right) \right)^{q}\frac{1}{n}\right)^{\frac{1}{q}}\\
\le&  \left(\left\Vert g\right\Vert_{\XX}^p+  \left\Vert \xx_1^*\left(f_0\right)\xx_1\right\Vert_{\XX}^p\right)^{\frac{1}{p}} +\left(1+2\alpha_1^p\alpha_2^p\right)^{\frac{1}{p}}\left(\sum_{n\in \NN}\left(w\left(n\right)\gamma_n\left(g\right) \right)^{q}\frac{1}{n}\right)^{\frac{1}{q}}\\
\le& 2^{\frac{1}{p}}\left(\left\Vert g\right\Vert_{\XX}+\alpha_3\left\Vert f_0\right\Vert_{\XX}\right)+\left(1+2\alpha_1^p\alpha_2^p\right)^{\frac{1}{p}}\left(\sum_{n\in \NN}\left(w\left(n\right)\gamma_n\left(g\right) \right)^{q}\frac{1}{n}\right)^{\frac{1}{q}}\\
\le& 2^{\frac{1}{p}}\left(1+2\alpha_1^p\alpha_2^p\right)^{\frac{1}{p}} n_0+2^{\frac{1}{p}}\alpha_3\left\Vert f_0\right\Vert_{\XX}\le n_1. 
\end{align*}
This completes the inductive step, so we have proved that there are $r,M>0$ such that
\begin{align*}
\overline{B}_{\left\Vert \cdot\right\Vert_{\circ}}\left(0,r\right)\cap SG\left(\XB\right)\cap \langle \XB\rangle&\subset U_{M}.
\end{align*}
Hence, if $f\in SG\left(\XB\right)\cap \langle \XB\rangle$, then
\begin{align*}
&\left\Vert f\right\Vert_{\GG^{w}_{q}}\le \frac{M}{r}\left\Vert f\right\Vert_{\circ}.
\end{align*}
For $q=\infty$, the argument goes through with only straightforward simplifications.  
\end{proof}
Our next Lemma is a Chebyshev version of the previous one; it is proved with a modification of the previous argument, so we shall be brief. 

\begin{lemma}\label{lemmasortofclosed2}Let $\XX$ be a $p$-Banach space $\XX$ with a basis $\XB$, $0< q\le \infty$, and $w$ a weight. For each $t>0$, define
\begin{align*}
V_t=V_t\left[\CG^{w}_q\right]:=\left\lbrace f\in  \CG^{w}_{q}: \left\Vert f\right\Vert_{\CG^{w}_{q}}\le t \right\rbrace. 
\end{align*}
Then 
\begin{enumerate}[\rm (i)]
\item \label{firstpartforbaire2} For each $t>0$ we have
\begin{align}
&\left\Vert f\right\Vert_{\XX}+\left(\sum_{n\in NSG\left(f\right)}\left(w\left(n\right)\vartheta_n\left(f\right) \right)^{q}\frac{1}{n}\right)^{\frac{1}{q}}\le t&&\forall f\in \overline{V_t}^{\left\Vert \cdot\right\Vert_{\XX}} \label{strictgreedy2}
\end{align}
if $0<q<\infty$, and 
\begin{align}
\left\Vert f\right\Vert_{\XX}+\sup_{n\in NSG\left(f\right)}w\left(n\right)\vartheta_n\left(f\right) \le t&&\forall f\in \overline{V_t}^{\left\Vert \cdot\right\Vert_{\XX}}\label{strictgreedy2b}
\end{align}
if $q=\infty$. \\
Hence, if $f\in \langle\XB\rangle\cap   SG\left(\XB\right)\cap \overline{V_t}^{\left\Vert \cdot\right\Vert_{\XX}}$, then $f\in V_t$. 
\item \label{secondpartbaire2}Let  $\YY=\CG^{w}_q$.   Suppose that 
\begin{itemize}
\item $\YY$ is a linear space.
\item There is a quasi-norm $\left\Vert \cdot\right\Vert_{\circ}$ defined on $\YY$ such that $\left\Vert f\right\Vert_{\XX}\lesssim \left\Vert f\right\Vert_{\circ}$ and $\left(\YY, \left\Vert \cdot\right\Vert_{\circ}\right)$ is complete.
\item $\XB$ is bounded with respect to $\left\Vert \cdot\right\Vert_{\circ}$. 
\end{itemize} 
Then, there is $t>0$ such that 
\begin{align}
&\left\Vert f\right\Vert_{\CG^{w}_{q}}\le t\left\Vert f\right\Vert_{\left\Vert \cdot\right\Vert_{\circ}}&&\forall f\in \langle \XB\rangle \cap SG\left(\XB\right). \label{keybound22}
\end{align}
\end{enumerate}
\end{lemma}
\begin{proof}
\ref{firstpartforbaire2}  
\ref{firstpartforbaire2}  
To prove \eqref{strictgreedy2} and \eqref{strictgreedy2b}, we may assume that $NSG\left(f\right)\not=\emptyset$.   Choose $\left(f_k\right)_{k\in \NN}\subset V_t$ so that 
\begin{align*}
f_k\xrightarrow[k\to \infty]{\left\Vert \cdot\right\Vert_{\XX}}f,
\end{align*}
and fix $m\in \NN$ so that $NSG\left(f\right)_{\le m}\not=\emptyset$. Given $n\in NSG\left(f\right)_{\le m}$, let $A$, $\delta$ and $k_0$ be as in the proof of Lemma~\ref{lemmasortofclosed}\ref{firstpartforbaire}. As 
\begin{align*}
&\vartheta_n\left(f\right)=\min_{g \in \langle \xx_n: n\in A\rangle}\left\Vert f-g\right\Vert, 
\end{align*}
it follows that 
\begin{align*}
\vartheta_n\left(f_k\right)\xrightarrow[k\to \infty]{}\vartheta_n\left(f\right),  
\end{align*}
and the proof of \eqref{strictgreedy2} (and \eqref{strictgreedy2b}) is completed by the argument of Lemma~\ref{lemmasortofclosed}\ref{firstpartforbaire}. \\
Now fix $0<q<\infty$, and pick $0\not=f_0\in \langle\XB\rangle\cap   SG\left(\XB\right)\cap \overline{V_t}^{\left\Vert \cdot\right\Vert_{\XX}}$. Let $B_0:=\supp\left(f_0\right)$ and $m_0:=\left\vert B_0\right\vert$. Note that if $n>m_0$, then $\vartheta_n\left(f_0\right)=0$. On the other hand, if $1\le n\le m_0$, then $n\in NSG\left(f_0\right)$ by Remark~\ref{remarkonlyonegreedy}. Hence, by \eqref{strictgreedy2}, 
 \begin{align*}
&\left\Vert f_0\right\Vert_{\XX}+\left(\sum_{n\in \NN}\left(w\left(n\right)\vartheta_n\left(f\right) \right)^{q}\frac{1}{n}\right)^{\frac{1}{q}}=\left\Vert f_0\right\Vert_{\XX}+\left(\sum_{n\in NSG\left(f_0\right)}\left(w\left(n\right)\vartheta_n\left(f\right) \right)^{q}\frac{1}{n}\right)^{\frac{1}{q}}\le t. 
 \end{align*}
 The proof of the case $q=\infty$ is similar.\\
\color{black}
\ref{secondpartbaire2}
We proceed as in the proof of Lemma~\ref{lemmasortofclosed}\ref{secondpartbaire} substituting the $V_r$'s for $U_r$'s, until we obtain $0<r_0\le 1$ and $f_0\in \langle \XB\rangle\cap SG\left(\XB\right)$ such that
\begin{align}
\overline{B}_{\left\Vert \cdot\right\Vert_{\circ}}\left(f_0,r_0\right)\cap \langle \XB\rangle \cap SG\left(\XB\right)\subset V_{n_0}. \label{onebaire5}
\end{align}
As before, we assume that $A:=\supp\left(f_0\right)\not=\emptyset$ and $SGS\left(f_0,n\right)=\left\lbrace \left\lbrace 1,\dots,n\right\rbrace\right\rbrace$ for each $1\le n\le m_0:=\left\vert A\right\vert$, and we want to find  $0<r\le r_0$ and $M\ge n_0$ so that when we substitute $r$ for $r_0$ and $M$ for $n_0$ in \eqref{onebaire5}, we may also substitute $0$ for $f_0$. By induction, we just have to find $0<r_1\le r_0$ and $n_1\ge n_0$ so that when we replace $r_0$ and $n_0$ by $r_1$ and $n_1$ respectively, we may also replace $f_0$ with some $f_1\in SG\left(\XB\right)\cap V_{n_1}\cap \langle \XB\rangle$ whose support has cardinality no greater $m_0-1$. Let
\begin{align*}
n_1:=&\floor{ 2^{\frac{1}{p}}\left(\alpha_2^p\alpha_1^p+\alpha_3^p+1\right)^{\frac{1}{p}} n_0+2^{\frac{1}{p}}\alpha_3\left\Vert f_0\right\Vert_{\XX}}+2,\\
f_1:=&f_0-\xx_1^*\left(f_0\right)\xx_1,\\
A_1:=&A\setminus \left\lbrace 1\right\rbrace,\\
r_1:=&\frac{r_0}{2\left(1+\alpha_2\right)}\min\left\lbrace 1,\min_{\substack{n\in A\\j\in \NN\setminus\left\lbrace n\right\rbrace}}\left\lbrace \left\vert \left\vert\xx_n^*\left(f_0\right)\right\vert-\left\vert\xx_j^*\left(f_0\right)\right\vert\right\vert \right\rbrace\right\rbrace.
\end{align*}
To complete the inductive step, it suffices to prove that 
\begin{align}
\overline{B}_{\left\Vert \cdot\right\Vert_{\circ}}\left(f_1,r_1\right)\cap SG\left(\XB\right)\cap \langle \XB\rangle\subset V_{n_1}. \label{keyclaim3.0}
\end{align}
To see this, fix $$0\not=f\in \overline{B}_{\left\Vert \cdot\right\Vert_{\circ}}\left(f_1,r_1\right)\cap SG\left(\XB\right)\cap \langle \XB\rangle.$$ If $2\le n\le  m_0$ and $k>n$ or $k=1$, 
\begin{align*}
\left\vert \xx_n^*\left(f\right)\right\vert\ge& \left\vert \xx_n^*\left(f_1\right)\right\vert-\left\vert \xx_n^*\left(f-f_1\right)\right\vert\ge \left\vert \xx_n^*\left(f_1\right)\right\vert-\alpha_2r_1\\
\ge&  2\left(1+\alpha_2\right)r_1+ \left\vert \xx_k^*\left(f_1\right)\right\vert-\alpha_2r_1\\
>&\left(1+\alpha_2\right)r_1+\left\vert \xx_k^*\left(f\right)\right\vert-\left\vert \xx_k^*\left(f_1-f\right)\right\vert>\left\vert \xx_k^*\left(f\right)\right\vert.
\end{align*}
Hence, for each $2\le n\le m_0$,
$$GS\left(f,n-1\right)=SGS\left(f,n-1\right)=\left\lbrace \left\lbrace 2,\dots, n\right\rbrace\right\rbrace.$$ Similarly, if $g=f+\xx_1^*\left(f_0\right)\xx_1$, for each $1\le n\le  m_0$ we have $GS\left(g,n\right)=SGS\left(g,n\right)=\left\lbrace \left\lbrace 1,\dots, n\right\rbrace\right\rbrace$. Thus, if $j\in \NN$ and $D\in GS\left(f,j\right)$, we consider the possible cases: 
\begin{itemize}
\item $j\ge \left\vert \supp\left(f\right)\right\vert$, then $\vartheta_j\left(f\right)=0$. 
\item If $1\le j<  \left\vert \supp\left(f\right)\right\vert$ and $1\in D$, then $GS\left(g,j\right)=GS\left(f,j\right)=\left\lbrace D\right\rbrace$. Hence, there is $G\in \langle \xx_n: n\in D\rangle$ such that $\vartheta_j\left(g\right)=\left\Vert g-G\right\Vert_{\XX}$. Pick $F\in \langle \xx_n: n\in D\rangle$ so that $\left\Vert f-F\right\Vert_{\XX}=\left\Vert g-G\right\Vert_{\XX}$. We have 
\begin{align*}
\vartheta_j\left(f\right)\le &\left\Vert f-F\right\Vert_{\XX}=\vartheta_j\left(g\right).
\end{align*}
\item If $1\le j<  \left\vert \supp\left(f\right)\right\vert$ and $1\not\in D$, then   $GS\left(f,j\right)=\left\lbrace D\right\rbrace$ and there is $k\in D$ such that if $E=\left\lbrace 1\right\rbrace\cup D\setminus\left\lbrace k\right\rbrace$, then $GS\left(g,j\right)=\left\lbrace E\right\rbrace$. 
Choose $G\in \langle \xx_n: n\in E\rangle$ so that $\vartheta_j\left(g\right)=\left\Vert g-G\right\Vert$, and let $F:=P_D\left(G\right)$. We have
\begin{align*}
\vartheta_j^p\left(f\right)\le& \left\Vert f-F\right\Vert_{\XX}^p\le \left\vert \xx_1^*\left(f\right)\right\vert^p \left\Vert\xx_1\right\Vert_{\XX}^p+ \left\Vert f-\xx_1^*\left(f\right)\xx_1-F\right\Vert_{\XX}^p\\
\le& \left\vert \xx_k^*\left(f\right)\right\vert^p \alpha_1^p+ \left\Vert f-\xx_1^*\left(f\right)\xx_1-F\right\Vert_{\XX}^p\\
=&  \left\vert \xx_k^*\left(g-G\right)\right\vert^p \alpha_1^p +\left\Vert g-G-\xx_1^*\left(g-G\right)\xx_1\right\Vert_{\XX}\\
\le& \left(\alpha_2^p\alpha_1^p+\alpha_3^p+1\right)\left\Vert g-G\right\Vert^p.
\end{align*}
It follows that 
\begin{align*}
\vartheta_j\left(f\right)\le&  \left(\alpha_2^p\alpha_1^p+\alpha_3^p+1\right)^{\frac{1}{p}}\vartheta_j \left(g\right).
\end{align*}
\end{itemize}
Since $g\in \overline{B}_{\left\Vert \cdot\right\Vert_{\circ}}\left(f_0,r_0\right)\cap SG\left(\XB\right)\cap\langle \XB\rangle\subset V_{n_0}$, in the case $0<q<\infty$ we have
\begin{align*}
\left\Vert f\right\Vert_{\CG^{w}_{q}}=&\left\Vert f\right\Vert_{\XX}+\left(\sum_{n\in \NN}\left(w\left(n\right)\vartheta_n\left(f\right) \right)^{q}\frac{1}{n}\right)^{\frac{1}{q}}\\
\le&  \left(\left\Vert g\right\Vert_{\XX}^p+  \left\Vert \xx_1^*\left(f_0\right)\xx_1\right\Vert_{\XX}^p\right)^{\frac{1}{p}} +\left(\alpha_2^p\alpha_1^p+\alpha_3^p+1\right)^{\frac{1}{p}}\left(\sum_{n\in \NN}\left(w\left(n\right)\vartheta_n\left(g\right) \right)^{q}\frac{1}{n}\right)^{\frac{1}{q}}\\
\le& 2^{\frac{1}{p}}\left(\left\Vert g\right\Vert_{\XX}+\alpha_3\left\Vert f_0\right\Vert_{\XX}\right)+\left(\alpha_2^p\alpha_1^p+\alpha_3^p+1\right)^{\frac{1}{p}}\left(\sum_{n\in \NN}\left(w\left(n\right)\vartheta_n\left(g\right) \right)^{q}\frac{1}{n}\right)^{\frac{1}{q}}\\
\le& 2^{\frac{1}{p}}\left(\alpha_2^p\alpha_1^p+\alpha_3^p+1\right)^{\frac{1}{p}} n_0+2^{\frac{1}{p}}\alpha_3\left\Vert f_0\right\Vert_{\XX}\le n_1, 
\end{align*}
and the inductive step is complete. Hence, we have proved that there are $r,M>0$ such that
\begin{align*}
\overline{B}_{\left\Vert \cdot\right\Vert_{\circ}}\left(0,r\right)\cap SG\left(\XB\right)\cap \langle \XB\rangle&\subset V_{M}.
\end{align*}
Thus, if $f\in SG\left(\XB\right)\cap \langle \XB\rangle$, then
\begin{align*}
&\left\Vert f\right\Vert_{\CG^{w}_{q}}\le \frac{M}{r}\left\Vert f\right\Vert_{\circ}.
\end{align*}
For $q=\infty$, the argument goes through with only straightforward simplifications.  
\end{proof}

The following lemma gives characterizations of superdemocratic bases that we will use in the proof of
Propositions~\ref{propositionUCCdemocratic} and~\ref{propositionTQGdemocratic}. 
\begin{lemma}\label{lemmademsuperdemSUCC}Let $\XB$ be a basis $p$-Banach space $\XX$. Suppose there are $C_1>0$, $n_1\in \NN$ and a subsequence $\left(\xx_{n_l}\right)_{l\in \NN}$ with the following property: Given $m\in\NN$, there is $l\left(m\right)\in \NN_{>m}$ such that if 
$A\subset \left\lbrace 1,\dots, m\right\rbrace$, and $B\subset \left\lbrace n_l: l>  l\left(m\right)\right\rbrace$, then for all $\varepsilon\in \EE^{A}$ we have
\begin{align*}
&\left\Vert \Ind_{\varepsilon, A}\right\Vert_{\XX}\le C_1 \left\Vert \Ind_{B}\right\Vert_{\XX} &&\text{ if }\left\vert B\right\vert=n_1\left\vert A\right\vert;\\
&\left\Vert \Ind_{B}\right\Vert_{\XX}\le C_1 \left\Vert \Ind_{\varepsilon, A}\right\Vert_{\XX} &&\text{ if }\left\vert A\right\vert=n_1\left\vert B\right\vert.
\end{align*}
Then, $\XB$ is $C$-superdemocratic with 
\begin{align*}
C=& \left(2\left(2n_1+1\right)\left(1+\alpha_1^p\alpha_2^pn_1\right)\right)^{\frac{1}{p}}C_1^2n_1^{\frac{1}{p}}.
\end{align*}
In particular, the above conditions hold if there are are $C_1>0$, $n_1\in \NN$ such that, whenever $A<B$ and $\varepsilon\in \EE^A$,  we have 
\begin{align*}
&\left\Vert \Ind_{\varepsilon, A}\right\Vert_{\XX}\le C_1 \left\Vert \Ind_{B}\right\Vert_{\XX} &&\text{ if }\left\vert B\right\vert=n_1\left\vert A\right\vert;\\
&\left\Vert \Ind_{B}\right\Vert_{\XX}\le C_1 \left\Vert \Ind_{\varepsilon, A}\right\Vert_{\XX} &&\text{ if }\left\vert A\right\vert=n_1\left\vert B\right\vert 
\end{align*}
\end{lemma}
\begin{proof}
Fix $A, B\in \NN^{<\infty}$ with $|A|\le |B|$, and $\varepsilon,\varepsilon' \in \EE^{\NN}$. If $|A|< n_1$, then 
\begin{align*}
\left\Vert \Ind_{\varepsilon, A}\right\Vert_{\XX}\le& (n_1-1)^{\frac{1}{p}}\alpha_1\alpha_2\left\Vert \Ind_{\varepsilon', B}\right\Vert_{\XX}.
\end{align*}
Otherwise, choose $B_1\subset B$ and $k_1\in \NN$ so that 
\begin{align*}
&|B_1|=k_1n_1; &&|B|<(k_1+1)n_1;&&B_1<B\setminus B_1. 
\end{align*}
Then there is a partition $\{A_j\}_{1\le j\le 2n_1+1}$ of $A$ with $0\le |A_j|\le k_1$ for each $1\le j\le 2n_1+1$ (note that some of the $A_j$'s may be empty). For each $1\le j\le 2n_1+1$, pick $D_j\supset A_j$ so that
\begin{align*}
&& |D_j\cup A_j|=k_1;&&& A_j<D_j\setminus A_j,
\end{align*}
and pick $\epsilon_j\in \{-1,1\}$ so that 
\begin{align*}
\left\Vert \Ind_{\varepsilon,A_j}\right\Vert_{\XX}\le& 2^{\frac{1}{p}}\left\Vert \Ind_{\varepsilon,A_j}+\epsilon_j\Ind_{D_j}\right\Vert_{\XX}.
\end{align*}
By the $p$-triangle inequality, we have
\begin{align*}
\left\Vert \Ind_{\varepsilon,A}\right\Vert_{\XX}\le& \left(2\left(2n_1+1\right)\right)^{\frac{1}{p}}\max_{1\le j\le 2n_1+1}\left\Vert \Ind_{\varepsilon,A_j}+\epsilon_j\Ind_{D_j}\right\Vert_{\XX}\\
=& \left(2\left(2n_1+1\right)\right)^{\frac{1}{p}}\left\Vert \Ind_{\varepsilon,A_{j_1}}+\epsilon_{j_1}\Ind_{D_{j_1}}\right\Vert_{\XX}.
\end{align*}
for some $1\le j_1\le 2n_1+1$. On the other hand, since $\left\vert B\setminus B_1\right\vert\le n_1$, 
\begin{align*}
&\left\Vert \Ind_{\varepsilon', B_1}\right\Vert_{\XX}\le \left(\left\Vert \Ind_{\varepsilon', B}\right\Vert_{\XX}^p+\alpha_1^p\alpha_2^pn_1\left\Vert \Ind_{\varepsilon', B}\right\Vert_{\XX}^p\right)^{\frac{1}{p}}=\left(1+\alpha_1^p\alpha_2^pn_1\right)^{\frac{1}{p}}\left\Vert \Ind_{\varepsilon', B}\right\Vert_{\XX}.
\end{align*}
Thus, to finish the proof we only need to prove that 
\begin{align*}
\left\Vert \Ind_{\varepsilon, A_{j_1}}+\epsilon_{j_1} \Ind_{D_{j_1}}\right\Vert_{\XX} \le & C_1^2n_1^{\frac{1}{p}} \left\Vert \Ind_{\varepsilon', B_1}\right\Vert_{\XX}.
\end{align*}
To this end, let $m:=\max\left(B\cup A_{j_1}\cup D_{j_1}\right)$, and choose disjoint sets $\left\lbrace E_{j}\right\rbrace_{1\le j\le n_1}$ so that, for each $j$, $\left\vert E_j\right\vert=k_1$ and $E_j\subset\left\lbrace n_l: l>l\left(m\right)\right\rbrace$. Let $E:=\bigcup_{j=1}^{n_1}E_j$. Then 
\begin{align*}
\left\Vert\Ind_{\varepsilon,A_{j_1}}+\epsilon_{j_1}\Ind_{D_{j_1}}\right\Vert_{\XX}\le& C_1\left\Vert \Ind_{E}\right\Vert_{\XX};\\
\left\Vert \Ind_{E_j}\right\Vert_{\XX}\le& C_1\left\Vert\Ind_{\varepsilon', B_1}\right\Vert_{\XX}&&\forall 1\le j\le n_1, 
\end{align*}
so the $p$-triangle inequality gives us the desired upper bound. \\
Finally, note that the `in particular'' follows by taking $n_l=l$ and $l\left(m\right)=m+1$ for all $l, m\in \NN$. 
\end{proof}

Next, we consider a property introduced in \cite{BL2020}, which will allow us to prove some results under weaker hypotheses in the case of Banach spaces. 

\begin{definition} (\cite[Definition 3.1]{BL2020}) \label{definitionseparation} 
Let $\XX$ be a quasi-Banach space. 
A sequence $\left(f_k\right)_{k\in\NN}\subseteq \XX$  has the \emph{finite dimensional separation property} with constant $C$ (or $C$-FDSP) if, for every separable subspace $\ZZ\subset \XX$ and every $\epsilon>0$, there is a Schauder basic subsequence $\left(u_{i_k}\right)_{k\in\NN}$ with basis constant no greater than $C+\epsilon$ for which the following holds: For every finite dimensional subspace $\YY \subset \ZZ$ there is $j_{\YY,\epsilon}\in \NN$ such that 
\begin{equation}
\left\Vert f\right\Vert_{\XX} \le \left(C+\epsilon\right)\left\Vert f+g\right\Vert_{\XX},\label{separation}
\end{equation}
for all $f\in \YY$ and all $g\in \overline{\langle u_{i_k}:k> j_{\YY, \epsilon}\rangle}$. We call any such subsequence a \emph{finite dimensional separating sequence} for $\left(\ZZ, M, \epsilon\right)$. 
\end{definition}

\begin{remark}\rm \label{remarkFDSP}
It is immediate that if one takes $\ZZ=\XX$ in Definition~\ref{definitionseparation}, every basis of $\XX$ that is (under some reordering) a Schauder basis of $\XX$ has the FDSP. In the case of Banach spaces, it is known that every Markushevich basis has the FDSP \cite[Proposition 2.11]{BL2020}, though a basis can have this property even if its dual basis is not total. In fact, in such spaces, a basis has the finite dimensional separation property if and only if it has a Schauder basic subsequence \cite[Corollary 3.9]{BL2020}.  \end{remark}

Our next result allows us to obtain $\A^{w}_q\approx \GG^{w}_{q}$ (and superdemocracy) from $\A^{w}_q=\GG^{w}_{q}$.

\begin{proposition}\label{propositionUCCdemocratic}Let $0<p\le 1$, $0<q\le \infty$, and $w\in \WW_{i,d,+}$. Suppose  that $\XB$ is a  basis of a $p$-Banach space $\XX$, for which $\A^{w}_q=\GG^{w}_{q}$, and one of the following conditions holds: 
\begin{enumerate}[\rm (i)]
\item \label{SUCC} $\XB$ is suppression unconditional for constant coefficients. 
\item \label{FDSP2} $\XB$ has the finite dimensional separation property. 
\end{enumerate}
Then, $\XB$ is superdemocratic, and $\A^w_q\approx \GG^w_q$. \\
In particular, \ref{FDSP2} holds if $\XB$ is (under some reordering) a Schauder basis, or if $p=1$ and $\XB$ is a Markushevich basis. 
\end{proposition}
\begin{proof}
First note that $v:=w^q\in \WW_{i,d,+}$. Choose $K_1>0$ so that $v$ is $K_1$-doubling. By scaling, we may assume $w(1)\ge 1$.  By \cite[Propositions 2.4,2.5]{BBGHO2018}, there is $K_2>0$ such that
\begin{align*}
K_2^{-1}\widetilde{v}(n) \le  v(n)\le K_2 \widetilde{v}(n) \qquad\forall n\in \NN. 
\end{align*}
Since $A^{w}_q$ is a quasi-Banach space, taking $\left\Vert \cdot\right\Vert_{\circ}=\left\Vert \cdot\right\Vert_{A^w_q}$ in Lemma~\ref{lemmasortofclosed} we obtain $K_3>0$ such that 
\begin{align}
 \left\Vert f\right\Vert_{\GG^{w}_q}\le K_3  \left\Vert f\right\Vert_{\A^{w}_q}\quad\forall f\in \langle \XB\rangle\cap SG\left(\XB\right), \label{boundforfinite}
\end{align}
 so $\A^w_q\approx \GG^w_q$. To prove that $\XB$ is superdemocratic,  we consider first the case where $\XB$ is suppression unconditional for constant coefficients, say with constant $K_4$.
By Theorem~\ref{theoremLRPURP}, $w$ has the LRP, so there are constants $K_{5}>0$ and $0<\beta<1$ such that 
$$
K_{5}\left(\frac{w(m)}{w(n)}\right)\ge \left(\frac{m}{n}\right)^\beta\qquad\forall m\ge n. 
$$
Let 
\begin{align*}
&K_6:=2^{2+\frac{1}{p}+\frac{1}{q}}K_1K_2^{\frac{1}{q}}&&\text{and}&&&k:=\floor{\left(4K_3K_{4}K_{2}^{\frac{1}{q}}K_{5}(K_6+1)\right)^{\frac{1}{\beta}}}+1.
\end{align*}
We claim that $\XB$ is $C_k$-superdemocratic, with 
\begin{align*}
C_k=& \left(2\left(2k+1\right)\left(1+\alpha_1^p\alpha_2^pk\right)\right)^{\frac{1}{p}}k^{2+\frac{1}{p}}.
\end{align*}
Suppose, to obtain a contradiction, that the claim is not true. Taking $C_1=k=n_1$ in Lemma~\ref{lemmademsuperdemSUCC}, we can find $A, B\in\NN^{<\infty}$ and $\varepsilon\in \EE^{\NN}$ such that 
\begin{align*}
&|B|=k|A|, &&A<B\text{ or }B<A,&&&\left\Vert \Ind_{\varepsilon, A}\right\Vert_{\XX} >k\left\Vert \Ind_{\varepsilon, B}\right\Vert_{\XX}
\end{align*}
(some of the $\varepsilon_n$'s can be taken to be $1$, but we do not need this). Choose $0<\epsilon<1$ and pick a finite sequence $\left(a_n\right)_{n\in A\cup B}\subset (0,\epsilon)$ that is strictly increasing if $B>A$ and strictly decreasing if $B<A$. Let $f:=\sum_{n\in A\cup B}\left(1+a_n\right)\varepsilon_n\xx_n$. Then, $f\in \langle \XB\rangle\cap SG\left(\XB\right)$ and, for every $1\le n\le |B|$, the only greedy set of $f$ of cardinality $n$ is contained in $B$.  Hence, for such $n$, 
\begin{align*}
\left(\gamma_n\left(f\right)\right)^p\ge&\min_{D\subset B}\left\Vert f-P_{D}\left(f\right) \right\Vert_{\XX}^p\\
\ge& \min_{D\subset B}\left\Vert \Ind_{\varepsilon, B\setminus D}+\Ind_{\varepsilon, A}\right\Vert_{\XX}^p-\sum_{j\in A\cup B}\left\vert a_j\right\vert^p \left\Vert \xx_j\right\Vert_{\XX}^p \\
\ge& -2|B|\left(\epsilon \alpha_1\right)^p+K_{4}^{-p}\left\Vert\Ind_{{\varepsilon}, A}\right\Vert_{\XX}^p,
\end{align*}
so for $\epsilon$ sufficiently small we get
\begin{align}
\left\Vert\Ind_{{\varepsilon}, A}\right\Vert_{\XX}\le& 2K_{4} {\gamma}_n\left(f\right), \label{1}
\end{align}
which entails that 
\begin{align}
 \left\Vert f\right\Vert_{\GG^{w}_q}\ge& \left(\sum_{n=1}^{|B|}\frac{\left(w(n){\gamma}_n\left(f\right)\right)^q}{n}\right)^{\frac{1}{q}}\ge \frac{ \left\Vert\Ind_{{\varepsilon}, A}\right\Vert_{\XX}}{2K_{4}}\widetilde{v}\left(|B|\right)^{\frac{1}{q}} \nonumber\\
 \ge& \frac{ \left\Vert\Ind_{{\varepsilon}, A}\right\Vert_{\XX} w\left(k|A|\right)}{2K_{4}K_{2}^{\frac{1}{q}}}\nonumber &&\text {if }0<q<\infty;\\
 \left\Vert f\right\Vert_{\GG^{w}_q}\ge& \sup_{1\le n\le |B|}w(n){\gamma}_n\left(f\right)\ge \frac{ \left\Vert\Ind_{{\varepsilon}, A}\right\Vert_{\XX} w\left(k|A|\right)}{2K_{4}}\nonumber&&\text{ if }q=\infty.
 \end{align}
Hence, for all $0<q\le \infty$,  
\begin{align}
 \left\Vert f\right\Vert_{\GG^{w}_q}\ge& \frac{ \left\Vert\Ind_{{\varepsilon}, A}\right\Vert_{\XX} w\left(k|A|\right)}{2K_{4}K_{2}^{\frac{1}{q}}} = \frac{ \left\Vert\Ind_{{\varepsilon}, A}\right\Vert_{\XX} w\left(|B|\right)}{2K_{4}K_{2}^{\frac{1}{q}}}.\label{2}
\end{align} 
On the other hand, for each $n\ge |A|$, choosing a sufficiently small $\epsilon$ we have 
\begin{align*}
\sigma_n\left(f\right)^p\le& \left\Vert\Ind_{{\varepsilon}, B}\right\Vert_{\XX}^p+|B|\alpha_1^p\epsilon^p \le 2\left\Vert\Ind_{{\varepsilon}, B}\right\Vert_{\XX}^p.
\end{align*}
Additionally, 
\begin{align*}
\left\Vert f\right\Vert_{\XX}^p\le& \epsilon^p\left(|B|+|A|\right)\alpha_1^p+ \left\Vert \Ind_{{\varepsilon}, B}\right\Vert_{\XX}^p+\left\Vert \Ind_{{\varepsilon}, A}\right\Vert_{\XX}^p\le 2 \left\Vert \Ind_{{\varepsilon}, A}\right\Vert_{\XX}^p
\end{align*}
Hence,
\begin{align}
\left\Vert f\right\Vert_{\A^{w}_q}\le&2^{\frac{1}{p}}\left\Vert \Ind_{{\varepsilon}, A}\right\Vert_{\XX}+\left(\sum_{n=1}^{|A|}\frac{\left(w(n) 2^{\frac{1}{p}}\left\Vert \Ind_{{\varepsilon}, A}\right\Vert_{\XX}\right)^{q}}{n}+\sum_{n=|A|+1}^{|A|+|B|}\frac{\left(w(n)2^{\frac{1}{p}}\left\Vert \Ind_{{\varepsilon}, B}\right\Vert_{\XX}\right)^{q}}{n}\right)^{\frac{1}{q}}\nonumber\\
\le&2^{\frac{1}{p}}\left\Vert \Ind_{{\varepsilon}, A}\right\Vert_{\XX}\left(1+2^{1+\frac{1}{q}}\widetilde{v}\left(|A|\right)^{\frac{1}{q}}\right)+2^{1+\frac{1}{q}+\frac{1}{p}}\left\Vert \Ind_{{\varepsilon}, B}\right\Vert_{\XX}\left(\widetilde{v}\left(|A|+|B|\right)\right)^{\frac{1}{q}}\nonumber\\
\le& 2^{2+\frac{1}{p}+\frac{1}{q}}K_{2}^{\frac{1}{q}}\left\Vert \Ind_{{\varepsilon}, A}\right\Vert_{\XX} w\left(|A|\right)+2^{1+\frac{1}{q}+\frac{1}{p}}\left\Vert \Ind_{{\varepsilon}, A}\right\Vert_{\XX}\frac{K_{1} K_{2}^{\frac{1}{q}}}{k}w\left(k|A|\right) \nonumber
\end{align}
if $0<q<\infty$, whereas 
\begin{align*}
\left\Vert f\right\Vert_{\A^{w}_{\infty}}\le&2^{\frac{1}{p}}\left\Vert \Ind_{{\varepsilon}, A}\right\Vert_{\XX}+\sup_{n=1}^{|A|}w(n)2^{\frac{1}{p}}\left\Vert \Ind_{{\varepsilon}, A}\right\Vert_{\XX}+\sup_{n=|A|+1}^{|A|+|B|}w(n)2^{\frac{1}{p}}\left\Vert \Ind_{{\varepsilon}, B}\right\Vert_{\XX}\nonumber\\
\le& 2^{\frac{2}{p}}K_1 \left\Vert \Ind_{{\varepsilon}, A}\right\Vert_{\XX}\left(w(|A|)+\frac{w(k|A|)}{k}\right).
\end{align*}
Hence, for all $0<q\le \infty$ we have 
\begin{align}
\left\Vert f\right\Vert_{\A^{w}_q}\le K_{6}\left\Vert \Ind_{{\varepsilon}, A}\right\Vert_{\XX}\left(w(|A|)+\frac{w(k|A|)}{k}\right).  \label{5}
\end{align}
From \eqref{5} and the fact that $|B|=k|A|$ we get 
\begin{align}
\left\Vert f\right\Vert_{A^{w}_q}\le& K_{6}\left\Vert \Ind_{{\varepsilon}, A}\right\Vert_{\XX} \left(K_{5}\frac{w(|B|)}{k^{\beta}}+\frac{w(k|A|)}{k}\right)\nonumber\\
\le& 2K_{6}(K_{5}+1)\left\Vert \Ind_{{\varepsilon}, A}\right\Vert_{\XX}  \frac{w(|B|)}{k^{\beta}} \label{6}
\end{align}
From \eqref{2} and \eqref{6} we obtain 
\begin{align*}
&K_3\ge \frac{ \left\Vert f\right\Vert_{\GG^{w}_q}}{ \left\Vert f\right\Vert_{\A^{w}_q}}\ge \frac{k^{\beta}}{4K_{4}K_{2}^{\frac{1}{q}}K_{6}(K_{5}+1)},
\end{align*}
which contradicts our choice of $k$. This proves the result when $\XB$ is unconditional for constant coefficients. \\
Note that if $\XB$ is a Schauder basis with constant $K_b$, if we set $K_4:=\left(K_b^p+1\right)^{\frac{1}{p}}$ the above proof holds with no changes. For the more general case in which $\XB$ only has the finite dimensional separation property, we just need the following, minor modifications: Let $\left(\xx_{n_j}\right)_{j\in \NN}$ be a separating sequence for $\XX, M, 1$ for some $M\ge 1$, and let $K_4:=\left(\left(M+1\right)^p+1\right)^{\frac{1}{p}}$; the other $K_j'$'s remain as before. To apply Lemma~\ref{lemmademsuperdemSUCC}, for each $m\in \NN$ let $\YY_m:=\langle \xx_n: 1\le n\le m\rangle$ and let $l\left(m\right):=j_{\YY_m,1}$. Now one can take $C_1=n_1=k$ in the lemma and define $C_k$ as before. If $\XB$ is not $C_k$-superdemocratic, there exists some $m\in \NN$, sets $D\subset \left\lbrace 1,\dots, m\right\rbrace$ and $E\subset \left\lbrace n_{l}: l> j_{\YY_{m},1}\right\rbrace$, and $\varepsilon\in \EE^{\NN}$ such that either
\begin{align}
&\left\vert D\right\vert=k\left\vert E \right\vert &&\text{and} &&&\left\Vert\Ind_{\varepsilon, E}\right\Vert_{\XX}>k  \left\Vert\Ind_{\varepsilon, D }\right\Vert_{\XX},\label{firstbig}
\end{align}
or 
\begin{align}
&\left\vert E\right\vert=k\left\vert D \right\vert &&\text{and} &&&\left\Vert\Ind_{\varepsilon, D }\right\Vert_{\XX}>k  \left\Vert\Ind_{\varepsilon, E}\right\Vert_{\XX}.\label{secondbig}
\end{align}
The rest of the proof is the same as that of \eqref{SUCC}, with $D$ taking the role of $B$ when \eqref{firstbig} holds and the role of $A$ when \eqref{secondbig} holds. Then one can check that the choice of $K_4$ and the separation property guarantee that the rest of the computations remain valid. 
\end{proof}
Under different hypotheses, we get a Chebyshev variant of Proposition~\ref{propositionUCCdemocratic}. 

\begin{proposition}\label{propositionTQGdemocratic}Let $0<p\le 1$, $0<q\le \infty$, and $w\in \WW_{i,d,+}$. Suppose  that $\XB$ is a  basis of a $p$-Banach space $\XX$ for which $\A^{w}_q=\CG^{w}_{q}$, and one of the following conditions holds: 
\begin{enumerate}[\rm (i)]
\item \label{TQG2} $\XB$ is truncation quasi-greedy.  
\item \label{FDSP22} $\XB$ has the finite dimensional separation property. 
\end{enumerate}
Then, $\A^w_q\approx \CG^w_q$, and $\XB$ is superdemocratic.
\end{proposition}
\begin{proof}
As before, we note that $v:=w^q\in \WW_{i,d,+}$, and choose $K_1>0$ so that $v$ is $K_1$-doubling. By scaling, we may assume $w(1)\ge 1$.  By \cite[Propositions 2.4,2.5]{BBGHO2018}, there is $K_2>0$ such that
\begin{align*}
K_2^{-1}\widetilde{v}(n) \le  v(n)\le K_2 \widetilde{v}(n) \qquad\forall n\in \NN. 
\end{align*}
Since $A^{w}_q$ is a quasi-Banach space, taking $\left\Vert \cdot\right\Vert_{\circ}=\left\Vert \cdot\right\Vert_{A^w_q}$ in Lemma~\ref{lemmasortofclosed2} we obtain $K_3>0$ such that 
\begin{align}
 \left\Vert f\right\Vert_{\CG^{w}_q}\le K_3  \left\Vert f\right\Vert_{\A^{w}_q}\quad\forall f\in \langle \XB\rangle\cap SG\left(\XB\right). \label{boundforfinite2}
\end{align}
We consider first the case where $\XB$ is truncation quasi-greedy. By \cite[Theorem 1.5]{AAB2024}, $\XB$ is $(1,1)$-bounded oscillation unconditional, say with constant $K_4$. By Theorem~\ref{theoremLRPURP}, $w$ has the LRP, so there are constants $K_{5}>0$ and $0<\beta<1$ such that 
$$
K_{5}\left(\frac{w(m)}{w(n)}\right)\ge \left(\frac{m}{n}\right)^\beta\qquad\forall m\ge n. 
$$
Let 
\begin{align*}
&K_6:=2^{2+\frac{1}{p}+\frac{1}{q}}K_1K_2^{\frac{1}{q}}&&\text{and}&&&k:=\floor{\left(4K_3K_{4}K_{2}^{\frac{1}{q}}K_{5}(K_6+1)\right)^{\frac{1}{\beta}}}+1.
\end{align*}
We claim that $\XB$ is $C_k$-superdemocratic, with 
\begin{align*}
C_k=& \left(2\left(2k+1\right)\left(1+\alpha_1^p\alpha_2^pk\right)\right)^{\frac{1}{p}}k^{2+\frac{1}{p}}.
\end{align*}
If this is not true, taking $C_1=k=n_1$ in Lemma~\ref{lemmademsuperdemSUCC}, we can find $A, B\in\NN^{<\infty}$ and $\varepsilon\in \EE^{\NN}$ such that 
\begin{align*}
&|B|=k|A|, &&A<B\text{ or }B<A,&&&\left\Vert \Ind_{\varepsilon, A}\right\Vert_{\XX} >k\left\Vert \Ind_{\varepsilon, B}\right\Vert_{\XX}
\end{align*}
Given $0<\epsilon<1$, pick a finite sequence $\left(a_n\right)_{n\in A\cup B}\subset (0,\epsilon)$ that is strictly increasing if $B>A$ and strictly decreasing if $B<A$. Let $f:=\sum_{n\in A\cup B}\left(1+a_n\right)\varepsilon_n\xx_n$. Then, $f\in \langle \XB\rangle\cap SG\left(\XB\right)$ and, for every $1\le n\le |B|$, the only greedy set of $f$ of cardinality $n$ is contained in $B$.  Hence, for such $n$, 
\begin{align*}
\left(\vartheta_n\left(f\right)\right)^p\ge&\inf_{ g\in \langle \xx_n: n\in B\rangle}\left\Vert f-g\right\Vert_{\XX}^p\\
\ge&\inf_{ g\in \langle \xx_n: n\in B\rangle}\left\Vert g+\Ind_{\varepsilon, A}\right\Vert_{\XX}^p-\sum_{j\in A\cup B}\left\vert a_j\right\vert^p \left\Vert \xx_j\right\Vert_{\XX}^p \\
\ge& -2|B|\left(\epsilon \alpha_1\right)^p+K_{4}^{-p}\left\Vert\Ind_{{\varepsilon}, A}\right\Vert_{\XX}^p,
\end{align*}
so for $\epsilon$ sufficiently small we get
\begin{align}
\left\Vert\Ind_{{\varepsilon}, A}\right\Vert_{\XX}\le& 2K_{4} {\vartheta_n\left(f\right)}, \label{12}
\end{align}
which entails that 
\begin{align}
 \left\Vert f\right\Vert_{\GG^{w}_q}\ge& \left(\sum_{n=1}^{|B|}\frac{\left(w(n){\vartheta}_n\left(f\right)\right)^q}{n}\right)^{\frac{1}{q}}\ge \frac{ \left\Vert\Ind_{{\varepsilon}, A}\right\Vert_{\XX}}{2K_{4}}\widetilde{v}\left(|B|\right)^{\frac{1}{q}} \nonumber\\
 \ge& \frac{ \left\Vert\Ind_{{\varepsilon}, A}\right\Vert_{\XX} w\left(k|A|\right)}{2K_{4}K_{2}^{\frac{1}{q}}}\nonumber &&\text {if }0<q<\infty;\\
 \left\Vert f\right\Vert_{\GG^{w}_q}\ge& \sup_{1\le n\le |B|}w(n){\gamma}_n\left(f\right)\ge \frac{ \left\Vert\Ind_{{\varepsilon}, A}\right\Vert_{\XX} w\left(k|A|\right)}{2K_{4}}\nonumber&&\text{ if }q=\infty.
 \end{align}
Hence, for all $0<q\le \infty$,  
\begin{align}
 \left\Vert f\right\Vert_{\GG^{w}_q}\ge& \frac{ \left\Vert\Ind_{{\varepsilon}, A}\right\Vert_{\XX} w\left(k|A|\right)}{2K_{4}K_{2}^{\frac{1}{q}}} = \frac{ \left\Vert\Ind_{{\varepsilon}, A}\right\Vert_{\XX} w\left(|B|\right)}{2K_{4}K_{2}^{\frac{1}{q}}}.\label{22}
\end{align} 
Now one can check that the rest of the proof of Proposition~\ref{propositionUCCdemocratic} holds for this case. 
\end{proof}

Now we can prove our main result. 

\begin{theorem}\label{theoremsqueeze2}Let $\XB$ be a basis of a $p$-Banach space $\XX$. If $\XB$ is truncation quasi-greedy and has the APP, the following are equivalent: 
\begin{enumerate}[\rm (i)]
\item \label{equivalences2} For each  $w\in \WW_{i,d,+}$ and every $0<q\le \infty$, 
\begin{align}
\GG^{w}_q\approx \A^{w}_q\approx \ell^q_{w\hr}. \label{characterization2}
\end{align}
\item \label{equality2} There is $w\in \WW_{i,d,+}$ and $0<q\le \infty$ such that $\GG^w_q=\A^w_q$. 
\item \label{equality2.2} There is $w\in \WW_{i,d,+}$ and $0<q\le \infty$ such that $\CG^w_q=\A^w_q$. 
\item \label{superdemocracy2}$\XB$ is superdemocratic. 
\item \label{democracy2}$\XB$ is democratic. 
\end{enumerate} 
\end{theorem}
\begin{proof}
By Proposition~\ref{propositionTQGdemocratic},  \ref{equality2.2} entails \ref{superdemocracy2}, whereas \ref{democracy2} implies \ref{equivalences2} by Lemma~\ref{lemmasqueeze1}. 
\end{proof}

\end{document}